\pdfoutput=1
\documentclass[12pt]{amsart}

\usepackage{amssymb}
\usepackage{amsmath}
\usepackage{amscd}
\usepackage{float}
\usepackage{graphicx}
\usepackage{amsfonts}
\usepackage{pb-diagram}
\usepackage{eufrak}
\usepackage[linktocpage=true]{hyperref}
\usepackage{cleveref} 
 
\usepackage[utf8]{inputenc}
\usepackage{tikz}
\usepackage{verbatim}
\usepackage[mathcal]{euscript}
\usepackage{float}
\usepackage{cite}
\usepackage[normalem]{ulem}
\usepackage{colortbl}
\usepackage{mathtools}
\usepackage{environ}            
 \usetikzlibrary{plotmarks}
\allowdisplaybreaks
\usetikzlibrary{decorations.pathreplacing}
\usepackage[toc,page]{appendix}

\newtheorem{thm}{Theorem}
\newtheorem{cor}{Corollary}
\newtheorem{lem}{Lemma}
\newtheorem{prop}{Proposition}
\newtheorem{defn}{Definition}
\newtheorem{rem}{Remark}

\newtheorem{quest}{Question}

\newcounter{Intro}

\newtheorem*{thm:MDKK}{Theorem \ref{thm:MDKK}}
\newtheorem*{MainThm:LeeHomology}{Theorem \ref{MainThm:LeeHomology}}
\newtheorem*{thm:UnorientedJonesISInvariant}{Theorem \ref{thm:UnorientedJonesISInvariant}}
\newtheorem*{thm:ComponentCount}{Theorem \ref{thm:ComponentCount}}
\newtheorem*{thm:CoreMantleInvariant}{Theorem \ref{thm:CoreMantleInvariant}}
\newtheorem*{thm:unoriented-link-invariant-homology}{Theorem \ref{thm:unoriented-link-invariant-homology}}

\theoremstyle{definition}
\newtheorem{ex}{Example}

\def\s{\mathfrak{s}}

\let \ttorg \tt \def \tt{\ttorg \obeyspaces}

\makeatletter

\long\def\M#1{\leavevmode\setbox\@tempboxa\hbox{#1}\@tempdima\fboxrule
    \advance\@tempdima \fboxsep \advance\@tempdima \dp\@tempboxa
   \hbox{\lower \@tempdima\hbox
  {\vbox{\hrule \@height \fboxrule
          \hbox{  \hskip\fboxsep
          \vbox{\vskip\fboxsep \box\@tempboxa\vskip\fboxsep}\hskip
                 \fboxsep\vrule \@width \fboxrule}%
                  }}}}
\makeatother

\newcommand{\BZ}{\mathbb{Z}}

\newcommand{\Across}{\raisebox{-0.25\height}{\includegraphics[width=0.5cm]{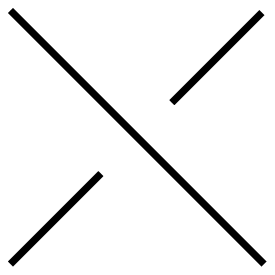}}}

\newcommand{\Asmooth}{\raisebox{-0.25\height}{\includegraphics[width=0.5cm]{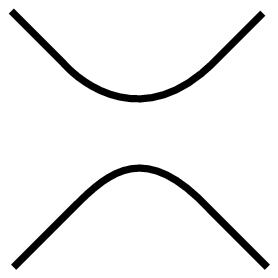}}}
\newcommand{\Bsmooth}{\raisebox{-0.25\height}{\includegraphics[width=0.5cm]{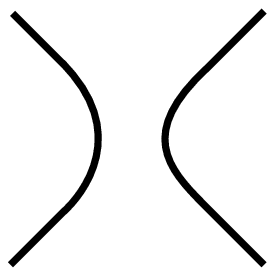}}}
\newcommand{\Rcurl}{\raisebox{-0.25\height}{\includegraphics[width=0.5cm]{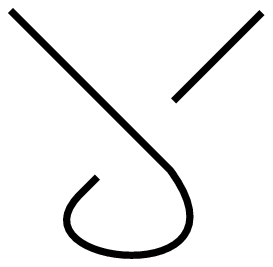}}}
\newcommand{\Lcurl}{\raisebox{-0.25\height}{\includegraphics[width=0.5cm]{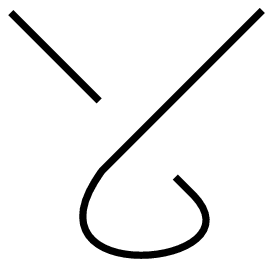}}}
\newcommand{\Arc}{\raisebox{-0.25\height}{\includegraphics[width=0.5cm]{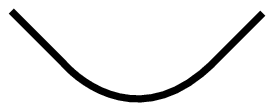}}}

\makeatletter

\long\def\UR#1{\leavevmode\setbox\@tempboxa\hbox{#1}\@tempdima\fboxrule
    \advance\@tempdima \fboxsep \advance\@tempdima \dp\@tempboxa
   \hbox{\lower \@tempdima\hbox
  {\vbox{\hrule \@height \fboxrule
          \hbox{  \hskip\fboxsep
          \vbox{\vskip\fboxsep \box\@tempboxa\vskip\fboxsep}\hskip
                 \fboxsep\vrule \@width \fboxrule}%
                  }}}}
\makeatother

\makeatletter

\long\def\LR#1{\leavevmode\setbox\@tempboxa\hbox{#1}\@tempdima\fboxrule
    \advance\@tempdima \fboxsep \advance\@tempdima \dp\@tempboxa
   \hbox{\lower \@tempdima\hbox
  {\vbox{
          \hbox{  \hskip\fboxsep
          \vbox{\vskip\fboxsep \box\@tempboxa\vskip\fboxsep}\hskip
                 \fboxsep\vrule \@width \fboxrule}%
                 \hrule \@height \fboxrule}}}}
\makeatother

\makeatletter

\long\def\UL#1{\leavevmode\setbox\@tempboxa\hbox{#1}\@tempdima\fboxrule
    \advance\@tempdima \fboxsep \advance\@tempdima \dp\@tempboxa
   \hbox{\lower \@tempdima\hbox
  {\vbox{\hrule \@height \fboxrule
          \hbox{\vrule \@width \fboxrule \hskip\fboxsep
          \vbox{\vskip\fboxsep \box\@tempboxa\vskip\fboxsep}\hskip
                 \fboxsep }%
                  }}}}
\makeatother

\makeatletter

\long\def\LL#1{\leavevmode\setbox\@tempboxa\hbox{#1}\@tempdima\fboxrule
    \advance\@tempdima \fboxsep \advance\@tempdima \dp\@tempboxa
   \hbox{\lower \@tempdima\hbox
  {\vbox{
          \hbox{\vrule \@width \fboxrule \hskip\fboxsep
          \vbox{\vskip\fboxsep \box\@tempboxa\vskip\fboxsep}\hskip
                 \fboxsep }%
                 \hrule \@height \fboxrule}}}}
\makeatother

\makeatletter
\def\l@subsection{\@tocline{2}{0pt}{2.5pc}{5pc}{}}
\def\l@subsubsection{\@tocline{2}{0pt}{5pc}{7.5pc}{}}
\makeatother

\begin{document}

\date{}

\title{\bf Unoriented Khovanov Homology}

\author{Scott Baldridge}
\address{Department of Mathematics, Louisiana State University\\
Baton Rouge, LA}
\email{sbaldrid@math.lsu.edu}

\author{Louis H. Kauffman}
\address{Department of Mathematics, Statistics and Computer Science\\ 851 South Morgan Street\\ University of Illinois at Chicago\\
Chicago, Illinois 60607-7045} 

\address{Department of Mechanics and Mathematics\\
Novosibirsk State University\\
Novosibirsk, Russia}
\email{kauffman@uic.edu}

\author{Ben McCarty}
\address{Department of Mathematical Sciences, University of Memphis\\
Memphis, TN}
\email{ben.mccarty@memphis.edu}

\maketitle

\thispagestyle{empty}

\begin{abstract}
The Jones polynomial and Khovanov homology of a classical link are invariants that depend upon an initial choice of orientation for the link.  In this paper, we give a Khovanov homology theory for unoriented virtual links.  The graded Euler characteristic of this homology is proportional to a similarly-defined unoriented Jones polynomial for virtual links, which is a new invariant in the category of  non-classical virtual links.  The unoriented Jones polynomial continues to satisfy an important property of the usual one: for classical or even virtual links, the unoriented Jones polynomial evaluated at one is two to the power of the number of components of the link. As part of extending the main results of this paper to non-classical virtual links, a new framework for computing integral Khovanov homology is described that can be efficiently and effectively implemented on a computer.  We define an unoriented Lee homology theory for virtual links based upon the unoriented version of Khovanov homology.
\end{abstract}

\section{Introduction}

It is well-known that the Kauffman bracket polynomial does not require the choice of an orientation on the link.  However, the Jones polynomial and Khovanov homology based upon it does.  The same is true for virtual links: The Khovanov homology defined in \cite{ArbitraryCoeffs, DKK} is an {\em oriented} virtual link invariant.  In this paper, we show how to get polynomial and homology invariants of the underlying link that are independent of the orientation.

In the original Khovanov invariants, the orientation was used to determine an overall grading shift in the bracket homology  (cf. \Cref{theorem:khovanov-homology} and \Cref{thm:MDKK}).  (The bracket homology is the ``categorification'' of the Kauffman bracket of a diagram.)  We introduce appropriate  grading shifts that also lead to an invariant homology that does not require $L$ to be oriented.  To describe these shifts, we  need three numbers for any virtual diagram: $s_+$, $s_-$, and $m$. The number $m$ is the number of {\em mixed-crossings} of the diagram, i.e., the number of classical crossings between different components of the link.  Each component can also have several {\em self-crossings}.  The self-crossings bifurcate into two types, $s_+$ and $s_-$, that correspond to the number of positive and negative self-crossings of the diagram.

Let $(C(D),\partial)$ be the bracket complex defined in \Cref{Khovanov-Homology} for a diagram $D$ of an unoriented link $L$. The unoriented Khovanov chain complex, $\tilde{C}(D)$, is a gradings-shifted version of the bracket complex:
$$\tilde{C}(D)=C(D)[-s_- - \frac12 m]\{s_+ - 2s_- - \frac12 m \}.$$ 
Here the first grading shift is for the homological grading and the second is for the $q$-grading of the bi-graded complex.  The homology of this complex, called  {\em unoriented Khovanov homology}, is an invariant of the link (cf. \Cref{thm:unoriented-link-invariant-homology}):

\begin{thm:unoriented-link-invariant-homology}  Let $L$ be an unoriented virtual link. The unoriented Khovanov homology, denoted $\widetilde{Kh}(L)$, can be computed from any virtual diagram of $L$ and is a virtual link invariant. 
\end{thm:unoriented-link-invariant-homology}

For a classical link, the unoriented Khovanov homology can be defined directly from an unoriented link diagram (cf. Section~\ref{Khovanov-Homology}).  For virtual links, the present definition of the chain complex requires an initial choice of an orientation for the link diagram.  The chain complex is, up to isomorphism, independent of the choice of orientation. In this sense our homology is  unoriented. It remains an open problem how to define the chain complex without choosing an orientation for the link diagram when the link is virtual.

\begin{quest}
For virtual links, can unoriented Khovanov homology be defined directly from an unoriented  link diagram without ever choosing an orientation? \label{quest:directly-from-diagram}
\end{quest}

The bracket homology and grading-shifts also leads to an unoriented version of Lee homology:

\begin{MainThm:LeeHomology}
Let $L$ be an unoriented virtual link. The unoriented Lee Homology, $Kh'(L)$, is an invariant of the link $L$. 
\end{MainThm:LeeHomology}

The proofs of \Cref{thm:unoriented-link-invariant-homology} and \Cref{MainThm:LeeHomology} start by calculating the bracket homology $H(D)$, of the bracket complex $C(D)$, associated to a link diagram $D$, and then shifting gradings appropriately to get an invariant of the link $L$ (cf. \Cref{Khovanov-Homology}  and \Cref{section:bracket-homology-for-virtual-links}).  Calculating this homology using the theoretical framework of \cite{ArbitraryCoeffs} and \cite{DKK} involves several choices not required of classical links: source-sink orientations, cut systems, base points, global orders, etc.  This theoretical framework is easy to describe to a {\em human} by decorating these choices on {\em pictures} of the link diagram and its hypercube of states.  However, since 2007 \cite{ArbitraryCoeffs}, there has not been a computer program for calculating this version of the homology because it is hard to encode these same choices into a computer data structure that can then be used to calculate the differential.  In this paper, we build a new theoretical framework for this homology around the PD notation of a link diagram. The PD notation  of a link diagram is a set of four-tuples that encodes the diagram by labeling the arcs of each component in increasing numerical order, and using these numbers to record each crossing.  For example, the PD notation for the virtual trefoil described in \Cref{section:bracket-homology-for-virtual-links} is $PD(vT)=\{ (4,3,1,2), (1,4,2,3)\}$. In our framework, the data structures needed to calculate the homology can be derived directly from the PD notation, making the problem amenable to programing for a computer. In fact, the authors together with Heather Dye and Aaron Kaestner have created a program to compute these homologies from the PD notation of a link using this framework.

When we speak of theoretical framework we include the formal method for representing virtual links. Each change in formal representation has its own properties, and the search for the best formal representation for eliciting computation has driven this research.

The following theorem says that the homology so obtained is invariant.  For a detailed description, see see \Cref{section:bracket-homology-for-virtual-links}.

\bigskip

\begin{thm:MDKK} \label{thm:MDKK} Given two PD notations, $PD(D)$ and $PD(D')$, for the same  unoriented virtual link $L$, the bracket homology of each will be isomorphic up to a grading shift.  In particular, for some numbers $a$ and $b$, 
$$H(PD(D))\cong H(PD(D'))[a]\{b\}.$$
\end{thm:MDKK}

\Cref{thm:MDKK} is not a direct translation of \cite{ArbitraryCoeffs, DKK} from a diagrammatic representation of virtual links into PD notation.  This new framework requires a careful synthesis of results from, in chronological order,  Bar-Natan \cite{DB2}, Manturov, Lipshitz/Sarkar \cite{LipSar}, Dye/Kaestner/Kauffman, and Kamada \cite{Kamada4}.   \Cref{thm:MDKK} is new and useful.

While the virtual link is unoriented, the PD notation provides an orientation  (see Question~\ref{quest:directly-from-diagram}). The theorem says that the bracket homology is independent of orientations up to a grading shift.  Furthermore, the bracket homology is isomorphic to Khovanov homology and unoriented Khovanov homology up to a grading shift.

An immediate consequence of \Cref{thm:unoriented-link-invariant-homology}  and \Cref{prop:partial-determines-grading} is that the graded Euler characteristic of the unoriented Khovanov homology defines a polynomial invariant that does not depend on the orientation of the link:
 $$\chi_q(\widetilde{Kh}(L)) = \sum_{i,j\in \frac12\BZ} (-1)^i q^j \dim(\widetilde{Kh}^{i,j}(L)).$$
One might reasonably call this the ``unoriented Jones polynomial'' for the link (see $J^0_L$ in \Cref{unoriented-jones-poly-for-virt-links}).  However, this choice has two substantial deficiencies.  First, the polynomial has complex, not real, coefficients.  Second, even for classical links, evaluating the polynomial at $q=1$ is not always positive as it is for the usual Jones polynomial. A significant amount of this paper is dedicated to correcting both of these deficiencies. It turns out that the  solution---cores and mantles---can be applied to fix other problems that have come up in the virtual link theory literature.  

We can get a hint of how to fix these deficiencies by looking at classical links.  For classical links, the required normalization is to multiply $\chi_q(\widetilde{Kh}(L))$ by $(-1)^{\lambda}$ where $\lambda =\sum_{i<j} Lk(K_i, K_j)$. (Note that this normalization does require choosing an orientation, but in \Cref{sec:unoriented-jones-poly-classical} we show that $(-1)^{\lambda}$ is independent of that choice.)  In \Cref{unoriented-jones-poly-for-virt-links}, we define a generalization of $\lambda$, denoted $\widetilde{\lambda}$, based upon a modified linking number between different components.  This number, which is equal to the original $\lambda$ when the link is classical, solves the first deficiency:

\begin{thm:UnorientedJonesISInvariant}
The unoriented Jones polynomial, defined by
$$\widetilde{J}_L(q) = (-1)^{\left(\tilde{\lambda} - s_- - \frac12 m \right)} q^{\left(s_+-2s_--\frac12 m\right)} \langle L\rangle,$$
is an unoriented virtual link invariant.  Moreover, $\widetilde{J}_L \in \mathbb{Z}[q^{-\frac12},q^{\frac12}]$.
\end{thm:UnorientedJonesISInvariant}

The definition of $\tilde{\lambda}$ also solves the second deficiency.  An {\em even link} is a virtual link in which there are an even number of classical mixed-crossings for every component in a diagram of the link (cf. \cite{Kamada, MWY}).  Even links are sometimes called 2-colorable links in the literature (cf.  \cite{Rushworth2}).  Since all virtual knots and classical links are even, one can often generalize theorems of virtual knots and classical links to even virtual links. The next theorem is new in the literature for the usual (oriented) Jones polynomial for {\em non-classical} virtual links $L$.  It also shows that $\tilde{\lambda}$ correctly addresses the second  issue for the unoriented Jones polynomial (see the following and \Cref{cor:ComponentCountforUnrientedJones}):

\begin{thm:ComponentCount}
If $L$ is an oriented virtual link with $\ell$ components, then 
\[\widetilde{J}_L(1) = J_L(1) = \begin{cases} 
      2^{\ell} & \text{if }L \text{ is even} \\
      0 & \text{if  }L\text{ is odd.} 
   \end{cases}
\]
The number, $J_L(1)$, is independent of the choice of orientation.
\end{thm:ComponentCount}

Thus, we prove that the usual (oriented) Jones polynomial, $J_L$, counts the number of components of a virtual link when the link is even, extending this well-known property for classical links, and that the unoriented Jones polynomial, $\widetilde{J}_L$, also preserves this property.
 
 \begin{rem}
 Aspects of \Cref{thm:UnorientedJonesISInvariant} and \Cref{thm:ComponentCount} have been known about classical links since Jones. Morton  gave a specific formula for the change in the Jones polynomial under change in orientation of a component \cite{Morton}. His argument was based on the skein relation and prior to the discovery of the bracket model. The bracket model and the Markov trace models make such formulae straightforward to deduce in the classical case. We have returned to this subject in the nontrivial case of virtuals and Khovanov homology.
 \end{rem}

In order to define $\widetilde{\lambda}$, we needed to introduce a new idea in virtual link theory that extends even link invariants to all virtual links:  the identification of the \emph{core} and \emph{mantle} of a virtual link, and more generally, a \emph{multi-core decomposition} (See \Cref{Subsec:EvenCore}).   A multi-core decomposition is the separation of a virtual link $L$ into a set of invariant sub-links, $L = C_1 \cup \ldots \cup C_n \cup M_n$, where each sub-link $C_i$ is even ($C_1$ is called {\em the} core), and the sub-link $M_n$ is either the empty link, or it is odd ($M_n$ is the final mantle).  

While there are a number of invariants for even virtual links in the literature, generalizing them to all virtual links has been elusive.  They fail to be invariants for odd virtual links because the definition of the invariant often depends heavily on each component having an even number of classical mixed-crossings.  By identifying an invariant even sub-link, the core, and eventually a set of invariant even sub-links in the multi-core decomposition, one can derive an invariant of odd links by applying the even link invariants to each even core in the decomposition.

\begin{thm:CoreMantleInvariant}
Any invariant of even links, $\Psi$, induces a tuple of invariants $(\Psi(C_1),\ldots, \Psi(C_n))$ for the multi-core decomposition of a virtual link $L$, and the tuple itself is an invariant. 
\end{thm:CoreMantleInvariant}

In the case of the unoriented Jones polynomial, the multi-core decomposition identifies a maximal set of maximal even sub-links on which the virtual link ``acts like'' an even link.  Thus, the unoriented Jones polynomial for odd links defined in this paper is the polynomial that is the ``closest to'' the  unoriented Jones polynomial for classical links.\\

\noindent {\bf Acknowledgements.}
Kauffman's work was supported by the 
Laboratory of Topology and Dynamics, 
Novosibirsk State University 
(contract no. 14.Y26.31.0025 
with the Ministry of Education and Science 
of the Russian Federation).  All three authors would like to thank William Rushworth for many helpful conversations and suggestions.

\section{Unoriented Jones Polynomial for Classical Links}\label{sect:unoriented-jones-poly}
\label{sec:unoriented-jones-poly-classical} 

We introduce the main ideas of this paper by reviewing and generalizing the oriented version of the Jones polynomial for classical links to the unoriented Jones polynomial.  Recall that to get the Jones polynomial, one multiplies the Kauffman bracket by a normalization factor  that depends on the writhe, i.e., $$J_L(q)=(-1)^{-n_-} q^{n_+ - 2n_-}\langle L\rangle$$ where $n_-$ and $n_+$ are the number of negative and positive (classical) crossings respectively\footnote{The Jones polynomial is usually written with $(-1)^{n_-}$ instead of $(-1)^{-n_-}$. These are equivalent since $n_-$ is an integer.  Later in this paper, we work with half-integer powers of $-1$  where the negative of a power is not equivalent: $(-1)^\frac12 \not= (-1)^{-\frac12}$.  In this context, $(-1)^{-n_-}$ is the correct normalization.} (cf. \cite{JO2,DB1,K}).  Here we are using the form of the Kauffman bracket given by 
\begin{eqnarray}
\langle \Across \rangle &=& \langle \Asmooth \rangle -q \langle \Bsmooth \rangle, \mbox{and} \label{eq:bracket-crossing}\\
\langle \bigcirc \cup D \rangle &=& (q^{-1}+q) \langle D\rangle. \label{eq:bracket-disjoint-circ}
\end{eqnarray}  
Clearly, the normalization  depends upon the way the components are oriented, but the Kauffman bracket itself does not depend upon a choice of orientation.  The new normalization factor is defined as follows.  Let $L=K_1\cup  \cdots \cup K_\ell$ be a link with $\ell$ components and let $D$ be a virtual diagram of $L$.  A {\em self-crossing} in $D$ is a classical crossing in which both under- and over-arcs of the crossing are from the same component. For a specific component $K_i$, the usual signs for each self-crossing of $K_i$ are the same for either orientation of $K_i$.   Let $s_+$ and $s_-$ be the total number of positive and negative self-crossings of a link $L$. 

The first Reidemeister move is covered by the self-crossing data, and the third Reidemeister move does not play a role in any normalization.  For the second Reidemeister move, the case of using the move on a single component $K_i$ is taken care of by the self-crossing data.  The final case is a Reidemeister two move for strands from two different components of a  link. Therefore, the remaining type of crossing to consider is a mixed-crossing:  Let $m$ be the total number of mixed-crossings of a diagram $D$.  Note that we do not assign a positive or negative sign to these crossings.

Applying the Kauffman bracket to two strands with one over the other, we see that
$$\langle \raisebox{-0.25\height}{\includegraphics[width=0.5cm]{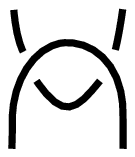}}\rangle =
-q \langle  \raisebox{-0.50\height}{\includegraphics[width=0.5cm]{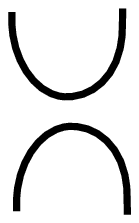}}\rangle.$$  
In other words, we can pull the strands apart at the cost of multiplying by an overall factor of $-q^{-1}$ (which partially motivates the usual normalization). Note that 
\begin{equation}
\label{eq:component-norm-factor}
(-1)^{-\frac12 m} q^{-\frac12 m},
\end{equation} 
yields the proper correction factor to ensure that the polynomial is invariant under a Reidemeister 2 move.   Further, note that the change in $-\frac12 m$ is equal to the change in $n_+-2n_-$ under the move because, for any orientation of the strands, there will always be one positive crossing and one negative crossing.

Thus we may define the {\em unoriented Jones polynomial} of a classical link $L$ by
\begin{equation}
\tilde{J}_L(q) = (-1)^{-n_-} q^{\left(s_+ - 2s_- - \frac12 m \right)} \langle L \rangle. \label{eq:Unoriented-Jones-Poly}
\end{equation}
Based upon  Expression~(\ref{eq:component-norm-factor}), one should expect to see $-s_- - \frac12 m$ as the exponent of $-1$ instead of $-n_-$.  However, as we will elaborate in \Cref{unoriented-jones-poly-for-virt-links}, 
\begin{equation}
n_- = - \lambda + s_- + \frac12 m, \label{eq:n_-}
\end{equation}
where $\lambda =\sum_{i<j} Lk(K_i, K_j)$.  To compute $\lambda$ requires a choice of orientation.  This sum of linking numbers changes by an even number when the orientation of $L$ is changed.  Hence, $n_-$ is made up of two terms that  do not change under a change of orientation and one term that changes by an even number. 

The discussion above immediately implies 

\begin{thm}
The unoriented Jones polynomial $\tilde{J}_L$ of a classical link $L$ is an invariant of the link.
\end{thm}

Next we extend the definition of the unoriented Jones polynomial to virtual links.  There are a number of pitfalls.  The first is that $\frac12 m$ and $Lk(K_i, K_j)$ can be half integers.  This means that the polynomial may have imaginary-valued coefficients.  Since we wish to preserve (as much as possible) the well known fact that evaluating the Jones polynomial at $1$ is $2^\ell$, this presents a problem.  Worse yet, if the orientation of $K_i$ (or $K_j$) is reversed, then $Lk(K_i, K_j)$ changes by an odd number if the total number of mixed-crossings between $K_i$ and $K_j$ is odd.  Hence,  the term $\lambda$ needs to be modified to compensate for this issue. We tried many potential modifications that were orientation invariant, however, these modifications could not be normalized so that the well-known fact ($2^\ell$) continued to hold, even for classical links.  The solution involves what we call the ``core'' of the virtual link, and it turns out that the core has far wider implications for virtual link theory.

\section{Unoriented Jones Polynomial for Virtual Links}
\label{unoriented-jones-poly-for-virt-links}

Before defining the  core of a virtual link, we begin this section by recalling some basic facts about virtual link theory.

\subsection{Virtual Knot Theory}\label{Subsec:vkt} Classical knot theory is the study of embeddings of disjoint unions of \( S^1 \) in \( S^3 \). Virtual knot theory, as introduced by Kauffman \cite{Kauffman1998}, is also the study of disjoint unions of $S^1$, but in a different ambient space:  \( \Sigma_g \times [ 0 , 1 ] \), where \( \Sigma_g \) denotes a closed orientable surface of genus \( g \).  Unlike links in many other \(3\)-manifolds, virtual links have a diagrammatic theory, akin to to that of classical links.  We begin by recalling some of the relevant facts about the theory, and refer the reader to \cite{Kauffman1998,Kauffman2020, ManturovVirtualKnots} for a more thorough treatment.

A virtual link diagram is a $4-$valent planar graph in which the vertices are decorated with classical crossings or \emph{virtual crossings}, denoted by \raisebox{-3pt}{\includegraphics[scale=0.35]{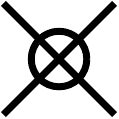}}.  Examples of virtual link diagrams are given in \Cref{Fig:virtualHopf} and \Cref{fig:borromean-rings}.

Two classical knots are equivalent if and only if their diagrams are related by a finite sequence of the classical Reidemeister moves.  Two virtual knots are equivalent if and only if their diagrams are related by a finite sequence of the  Reidemeister moves together with the virtual Reidemeister moves. The moves are shown in \Cref{Fig:vrms}.  One may think of the virtual link itself as an equivalence class of virtual link diagrams, each member of which is related to the rest by a finite sequence of classical and virtual Reidemeister moves.

\begin{figure}[H]
\includegraphics[scale=0.75]{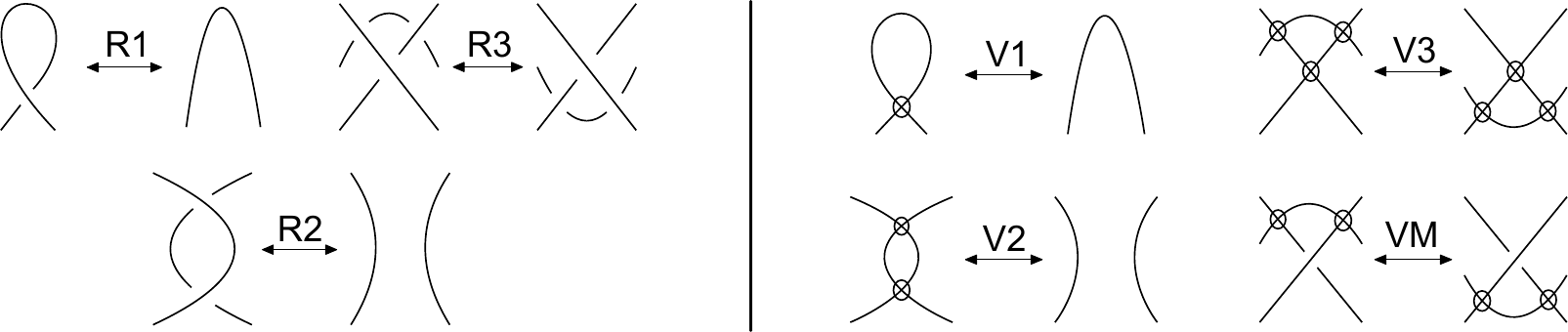}
\caption{The classical and virtual Reidemeister moves.}
\label{Fig:vrms}
\end{figure}

A classical link diagram is simply a virtual link diagram without virtual crossings. Thus, classical knot theory is a proper subset of virtual knot theory.  For an in-depth treatment of the diagrammatic theory of virtual links see \cite{Kauffman1998}.

\subsection{The Even Core of a Virtual Link}\label{Subsec:EvenCore}
For classical links, the Jordan Curve Theorem implies that each component of a link diagram intersects every other component of the diagram in an even number of crossings.  However, virtual crossings are not genuine crossings, which allows one to define various parities for virtual links.

\begin{defn}[Component-to-Component Parity] Let $D$ be a diagram of a virtual link $L$ with components $K_1,...,K_\ell$.  The {\em component-to-component parity}, denoted $\pi(K_i,K_j)$, is $0$ if there are an even number of mixed crossings between components $K_i$ and $K_j$ and $1$ otherwise.
\end{defn}

\begin{defn}[Component Parity]
\label{def:ComponentParity}
Let $D$ be a diagram of a virtual link $L$ with components $K_1,...,K_\ell$.  Define the parity of $K_i$, denoted $\pi(K_i)$, to be the number of mixed crossings of $D$ involving component $K_i$, modulo 2.  
\end{defn}

Note that the number of mixed crossings of a component is equal to the number of virtual crossings modulo two.

\begin{defn}[Link Parity]
\label{def:LinkParity}
Let $D$ be a diagram of a virtual link $L$ with components $K_1,...,K_\ell$.  The link $L$ is called {\em even} if all components of $L$ are even, i.e., $\pi(K_i)=0$ for all $i$, and called {\em odd} if there exists an odd component. The parity of $L$, denoted $\pi(L)$, is $\pi(L)=0$ if $L$ is even and $\pi(L)=1$ if $L$ is odd.
\end{defn}

Note that the virtual Reidemeister moves have no effect on these parities, and the classical Reidemeister moves clearly leave them unchanged, and hence we obtain:

\begin{lem}
\label{lem:parity}
The link parity, each component parity, and each component-to-component parity of a link are all virtual link invariants.
\end{lem}

It is well-known that even virtual links form a subset of virtual links for which it is often easier to define invariants (cf. \cite{Rushworth2, Kamada, MWY}).  However, odd virtual links present certain challenges.  Some of these challenges may be overcome by choosing an invariant even sublink.

\begin{defn}
\label{defn:EvenCore}
Given a diagram of a virtual link $L = K_1\cup K_2 \cup \ldots \cup K_\ell$ we obtain a sub-link $L_1$ by deleting all the odd components of $L$.  The resulting sub-link $L_1$ may be even or odd.  If it is even, stop. Otherwise, repeat the procedure on $L_1$ to get $L_2$, and continue until an even sub-link, $L_k$, is obtained ($L_k$ may be the empty link, which is even).  Call the even sub-link $L_k$  \emph{the core} of $L$ and denote it by $C$. 
\end{defn}

Note that after deleting the odd components of $L_i$ to get $L_{i+1}$, it is possible that components of $L_{i+1}$ that were even in $L_i$ become odd in $L_{i+1}$.  \Cref{lem:parity} still applies to the sub-link $L_{i+1}$ thought of as a link by itself.  Therefore the deletion process is unique: $L_1$ will always be even or odd as its own link, and the odd components of $L_1$ that are deleted to get $L_2$ will always be the same  components, and so on.  Thus, we immediately obtain:

\begin{lem}
\label{lem:EvenCoreInvariant}
Any invariant of the even core of a virtual link $L$ is an invariant of $L$.  
\end{lem}

Invariants of odd links calculated from the even core leave out much of the information about the link.  In order to recapture some of that information it is helpful to look at the complement of the even core, which we call the \emph{mantle} of the virtual link. 

\begin{defn}
\label{defn:Mantle}
The {\em mantle} of a link $L$ is the sub-link $M$ given by the complement of the core, i.e., $M=L\setminus C$.  
\end{defn}

Note that the mantle may be either an even or odd link in its own right, and possesses its own even core.  Hence, we may repeat the process described above.  Given a virtual link $L$, determine its core $C_1$ and mantle, $M_1$ ($L=M_1 \cup C_1$) using  \Cref{defn:EvenCore}. Next, determine the core of $M_1$ and denote it by $C_2$ using  \Cref{defn:EvenCore}.  This process also determines a new mantle, $M_2$.  At this stage, $L = C_1 \cup C_2 \cup M_2$.  Repeat this process until $M_n$ is either empty or has an empty core (if $M_n$ is non-empty and has empty core, it must contain only odd components).  Thus, one obtains a decomposition of the link:  $L = C_1 \cup \ldots \cup C_n \cup M_n$.  We call this the \emph{multi-core decomposition} of the link.  Repeated applications of \Cref{lem:EvenCoreInvariant} yield the following theorem.

\begin{thm}
\label{thm:CoreMantleInvariant}
Any invariant of even links, $\Psi$, induces a tuple of invariants $(\Psi(C_1),\ldots, \Psi(C_n))$ for the multi-core decomposition of a virtual link $L$, and the tuple itself is an invariant of $L$. 
\end{thm}

In particular, this theorem implies that any invariant of even links immediately generalizes to a new invariant of {\em all} virtual links---not just even links.  For example, the papers \cite{MWY} and \cite{Kamada} generalize the odd writhe of a virtual knot  \cite{Kauffman2004} to even virtual links.  In \cite{Kamada}, a virtual orientation is described where the orientation of a component changes direction in a diagram $D$ at every virtual crossing (which is why the link must be even).  Using this orientation, the sum of the signs of classical crossings  between $K_i$ and $K_j$ whose over arc is $K_i$ is denoted $\Lambda_D(i,j)$. Clearly, this number depends upon the order: $\Lambda_D(i,j) \not= \Lambda_D(j,i)$ in general.

\begin{cor}
\label{cor:virtualOrientationLinking}[cf. Theorem 14 of \cite{Kamada}, see also \cite{MWY}]
Let $L = K_1 \cup K_2 \cup \cdots \cup K_\ell$ be an ordered virtual link.  Every core $C_r$ inherits an ordering from the ordering of $L$.  For every core $C_r$ of $L$, and any pair of components, $K_i$ and $K_j$ in $C_r$, the number $|\Lambda_{C_r}(i,j)|$ is an invariant of the ordered unoriented virtual link $L$.
\end{cor}

Another generalization of the odd writhe to even links was given in \cite{Rushworth2}.  In that paper, an even link is equivalent to the link being 2-colorable.  The definition of 2-color writhe, $J^2(D)$, of an even link diagram $D$ depends upon a special parity function on each crossing that satisfies the parity axioms (briefly mentioned below).  For a given 2-coloring of an oriented link, define a quantity for that coloring as the sum of classical crossing signs of only the odd crossings (where odd is defined by that parity function).  The 2-color writhe of the link, $J^2(L)$, is then a tuple of these numbers---one for each 2-coloring of a diagram $D$ of $L$.  This tuple is defined up to permutations of the entries.

\begin{cor}
\label{cor:2colorwrithe}
The tuple consisting of the 2-color writhe of each core, $\left(J^2(C_1),\ldots, J^2(C_n)\right)$, is an invariant of the oriented virtual link $L$.
\end{cor}

Rushworth shows that for a virtual knot, $J^2(K)$ is the odd writhe of $K$ (see Proposition 3.4 of \cite{Rushworth2}).

Sometimes these tuples of invariants can be profitably added together to get an overall invariant of the link.  We want more, however.  For an upgraded definition of $\lambda$ for all virtual links (not just even links), we wish to define ``linking numbers'' for all components, not just components in each of the cores.  To do so, we introduce a parity function defined from pairs of components of $L$.  This parity function is defined on components, but could be described as a parity function on crossings (cf. \cite{ManturovParity},  and \cite{ManturovProjection}, \cite{Rushworth2}), as we will see next.

\subsection{A Parity Function}
It will be useful to define a parity function on pairs of components of a virtual link.  Let $L =  K_1 \cup \ldots \cup K_\ell$ be a virtual link with multi-core decomposition, $L = C_1 \cup C_2 \cup \ldots \cup C_n \cup M_n$.  We define a {\em component parity function} on pairs of components as follows:
\begin{equation}
p(K_i, K_j) = \left\{
        \begin{array}{ll}
            0 & \quad K_i, K_j\in C_r \,\, \text{for some} \, r \\
            1 & \quad \text{otherwise.}
        \end{array}
    \right.
\end{equation}

This component parity function descends to a parity function on crossings, as described in  \cite{ManturovParity} and \cite{Rushworth2}.   If $c$ is a crossing between component $K_i$ and $K_j$, then we assign the parity, $p(K_i,K_j)$, to $c$.  Any self-crossing is assigned a parity of $0$.  Axioms 0 and 1 (cf. \cite{Rushworth2}) are clearly satisfied.  Because all crossings between $K_i$ and $K_j$ will have the same parity, Axiom 2 is clearly satisfied.  For a Reidemeister 3 move, we either have all three components in the same core, $C_r$, two in the same core, and one is not, or, no two of the strands are in the same core.  These correspond to the three allowable cases of Axiom 3.  We could just work with the parity function on crossings, but find it convenient to work with the component parity function at the level of component-to-component for reasons that will become apparent below.

\subsection{The Unoriented Jones Polynomial of a Virtual Link}\label{Subsec:unorientVJones}

We wish now to generalize the unoriented Jones polynomial (\Cref{eq:Unoriented-Jones-Poly}) for classical links to an invariant of unoriented virtual links. 
In fact, the same definition together with \Cref{eq:n_-} works for any even virtual link.  Here is why: Suppose the orientation of $K_i$ is reversed.  In the classical case, the number of mixed crossings between $K_i$ and each other component will be even.  For an even virtual link, it is possible that the number of mixed crossings between $K_i$ and $K_j$ is odd for some $j$ (cf. \Cref{fig:borromean-rings}).  In that case, reversing the orientation on $K_i$ will change $Lk(K_i, K_j)$ by an odd number.  However, in an even virtual link, $\pi (K_i)=0$ for all $i$.  Hence, there must be an even number of components that intersect $K_i$ in an odd number of mixed crossings.  Thus, the overall parity of the exponent on $(-1)$ is unchanged by the orientation swap.

For an odd virtual link, the unoriented Jones polynomial as defined for even/classical links need not be an orientation invariant, as the following example shows. 

\begin{ex}
\label{ex:VirtualHopfs}
Let $L_0$ and $L_1$ be the oriented virtual Hopf links shown on the left and right, respectively in \Cref{Fig:virtualHopf}.  Observe that  unoriented Jones polynomial defined for classical links is not orientation invariant, because $\lambda- s_- -\frac12m$ for $L_0$ is $-1$ while for $L_1$ it is $0$. 
\end{ex}

\begin{figure}[H]
\includegraphics[scale=.6]{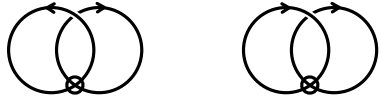}
\caption{Virtual Hopf links with two different orientations.}
\label{Fig:virtualHopf}
\end{figure}

Thus, we need to extend the  unoriented Jones polynomial defined for classical/even links to odd links.  This amounts to redefining the term $\lambda$.  We do this next.  Consider a diagram of an oriented virtual link $L = K_1 \cup \ldots \cup K_\ell$.

We first define a \emph{modified linking number}:
$$\widetilde{Lk}(K_i, K_j) \coloneqq (-1)^{p(K_i,K_j)\cdot \left(Lk(K_i,K_j) + \frac12 \pi(K_i,K_j)\right)} Lk(K_i,K_j),$$
where $p(K_i,K_j)$ is the component parity function and $\pi(K_i,K_j)$ is the component-to-component parity.  Suppose $L$ has  multi-core decomposition $L = C_1 \cup  \ldots \cup C_n \cup M_n$.  The modified linking number is, up to sign, the ordinary linking number, and if $K_i$ and $K_j$ belong to the same core, $C_r$, it has the same sign as the ordinary linking number.  If $K_i$ and $K_j$ belong to different cores, or if at least one of them belongs to the mantle, $M_n$, then the modified linking number may have the opposite sign as the ordinary linking number.  

\begin{lem}
\label{lem:modLinkingNumber}
The modified linking number, $\widetilde{Lk}(K_i, K_j)$, is an element of $\frac12 \mathbb{Z}$ and is an oriented virtual link invariant.
\end{lem}

\begin{proof}
Consider a diagram of an oriented link $L$, and a multi-core decomposition $L = C_1 \cup \ldots \cup C_n \cup M_n$.  If $K_i$ and $K_j$ belong to the same core, then the modified linking number is the ordinary linking number and is an integer since each core is an even sub-link.

If $K_i$ and $K_j$ do not belong to the same core, but they do interact in an even number of mixed-crossings, then $\widetilde{Lk}$ is still an integer.  Otherwise, $K_i$ and $K_j$ interact in an odd number of mixed-crossings.  Thus $Lk(K_i, K_j)$ is a half-integer, and because $\pi(K_i,K_j)= 1$ in this case, the exponent on $-1$ in the definition of $\widetilde{Lk}$ will be an integer.

The fact that $\widetilde{Lk}$ is an oriented virtual link invariant follows from the fact that the ordinary linking number and parity are link invariants.
\end{proof}

We can use the modified linking number to extend $\lambda$ from even links to all virtual links:
\begin{equation}
\label{eqn:lambda}
\tilde{\lambda}(L) = \sum_{1 \leq i < j \leq \ell} \widetilde{Lk}(K_i,K_j).  
\end{equation}

Since $\tilde{\lambda}$ is defined in terms of linking numbers, which are invariant under the Reidemeister moves (virtual and classical), we immediately obtain: 

\begin{lem}
\label{lem:lambdaInvariant}
The number $\tilde{\lambda}$ is an oriented virtual link invariant.
\end{lem}

\begin{rem}
\label{rem:relateLambda}
When $L$ is even, $L$ is the core. Thus, for an even link $L$, $\lambda=\tilde{\lambda}$ and $\tilde{\lambda}$ extends the definition of $\lambda$ to odd links. \label{rem:modified-equals-regular-linking-number}
\end{rem}

The number $\tilde{\lambda}$ is well-behaved under a change of orientation.  Suppose $K_s \in C_r$.  Let $L_s$ be the link $L$ with the same orientations on the components except the orientation of $K_s$ reversed.  Let $\overline{K_s}$ denote the component $K_s$ with the opposite orientation.  Subtracting $\tilde{\lambda}(L) - \tilde{\lambda}(L_s)$, we pick up only the terms where the orientation changes.  Thus:
\begin{eqnarray*}
\tilde{\lambda}(L) - \tilde{\lambda}(L_s) & = & \sum_{1\leq  j \leq \ell} \left( \widetilde{Lk}(K_s, K_j) - \widetilde{Lk}(\overline{K_s}, K_j)\right) \\
 & = & \sum_{1 \leq j \leq\ell} \left( (-1)^{p(K_s, K_j) \cdot \left( Lk(K_s, K_j) + \frac12 \pi(K_s,K_j)\right)} \right. \\
 & & \left. + (-1)^{p(K_s, K_j) \cdot \left(-Lk(K_s, K_j) + \frac12 \pi(K_s,K_j)\right)} \right) Lk(K_s,K_j).
\end{eqnarray*}
For any $j$ such that $K_j \in C_r$, the term in the sum above becomes $2Lk(K_s, K_j)$.  It may be that $K_s$ and $K_j$ interact in an odd number of classical crossings, in which case $2Lk(K_s, K_j)$ may be odd.  However, if that happens, it must do so for an even number of such $j$, because $C_r$ is an even sub-link.

For any $j$ such that $K_j \notin C_r$, the component  parity function is given by $p(K_s,K_j) = 1$, and so the exponents come into play.  It is still possible that $K_s$ and $K_j$ interact in an even or odd number of classical crossings.  If $K_s$ interacts with $K_j$ in an odd number of classical  crossings, then $\pi(K_s,K_j) = 1$ and one of the exponents above will be even and the other will be odd.  Hence that term will contribute $0$.  Otherwise, $\pi(K_s,K_j) = 0$ and the exponents will both be even or both be odd. Thus, that term will contribute $\pm 2Lk(K_s,K_j)$, which is an even number, since $Lk(K_s,K_j)$ is an integer in this case.  

If $K_r \in M_n$, the argument is similar to the previous case.  Hence, these observations, together with \Cref{lem:lambdaInvariant}, prove:

\begin{lem}
\label{lem:lambdaInvariant2}
If $L_s$ is the virtual link obtained from $L = K_1 \cup \ldots \cup K_\ell$ by reversing the orientation on component $K_s$ then $\tilde{\lambda}(L)$ and $\tilde{\lambda}(L_s)$ differ by an even integer.  
\end{lem}

By definition $\tilde{\lambda}$ is possibly a half-integer, and in many cases, an integer.  But, in either case the previous lemma guarantees that $(-1)^{\tilde{\lambda}}$ is invariant of the choice of orientation of $L$, since changing the orientation on a component changes $\tilde{\lambda}$ by an even integer. (Here and throughout the paper, we take $(-1)^\frac12 = {\rm i}$.)

\begin{thm}
\label{thm:lambdaInvariant3}
The complex number $(-1)^{\tilde{\lambda}}$ is invariant of the choice of orientation of $L$.
\end{thm}

We now have all of the ingredients in place to define the \emph{unoriented Jones polynomial} for any virtual link $L$.

\bigskip
\begin{defn}[Unoriented Jones Polynomial] The {\em unoriented Jones polynomial} for a virtual link $L$ is
\begin{equation}
\label{eqn:unorientedVirtualJones}
\widetilde{J}_L(q) = (-1)^{\left( \tilde{\lambda} - s_- - \frac12 m \right)} q^{\left(s_+-2s_--\frac12 m\right)} \langle L\rangle.
\end{equation}  
\end{defn}
\bigskip

Note that this definition extends both the definition of the unoriented Jones polynomial for classical links (cf.  \Cref{eq:Unoriented-Jones-Poly} and \Cref{eq:n_-}) and even links by \Cref{rem:modified-equals-regular-linking-number}.  When $L$ is an even virtual link (i.e. when $L$ is its own even core), these two formulas are identical.  

As observed in \Cref{thm:lambdaInvariant3}, $(-1)^{\tilde{\lambda}}$ need not be a real number, but it turns out that $(-1)^{\tilde{\lambda}-\frac12 m}$ is.  In particular, given a diagram of a virtual link $L = K_1 \cup \ldots \cup K_\ell$, observe that if $\pi(K_i,K_j) = 1$ then $\widetilde{Lk}(K_i,K_j)$ is a half-integer.  If there is an odd number of such pairs, then $\widetilde{\lambda}$ will be a half-integer as well.  Otherwise, $\widetilde{\lambda}$ will be an integer.  Similarly, in counting the total number of mixed crossings, if there is an odd number of pairs of components such that $\pi(K_i,K_j) = 1$, then there will be an odd number of mixed crossings in the diagram.  Hence, $\frac12 m$ will be a half-integer in this case, and will be an integer otherwise. In either case, noting that $s_-\in \mathbb{Z}$, we obtain the following lemma.

\begin{lem}
\label{lem:lambdaPlusMixed}
For any virtual link $L$, $\tilde{\lambda} - \frac12 m \in \mathbb{Z}$, and hence, $(-1)^{\tilde{\lambda} - s_- - \frac12 m} \in \{-1,1\}$.
\end{lem}

\begin{thm}
\label{thm:UnorientedJonesISInvariant}
The unoriented Jones polynomial, $\widetilde{J}_L$, is an unoriented virtual link invariant.  Moreover, $\widetilde{J}_L \in \mathbb{Z}[q^{-\frac12},q^{\frac12}]$.
\end{thm}

\begin{proof}
The Kauffman bracket clearly does not depend on orientation.  The normalization factor $q^{\left(s_+ -2s_- -\frac12 m\right)}$ depends only on self-crossings and the total number of mixed-crossings, neither of which change under an orientation switch, and $(s_- + \frac12 m)$ is orientation invariant for similar reasons. Thus, by \Cref{thm:lambdaInvariant3}, the first statement follows.

The second statement follows from \Cref{lem:lambdaPlusMixed} and the fact that $q^{\left(s_+-2s_--\frac12 m\right)}\langle L\rangle$ has integer coefficients.
\end{proof}

\subsection{Evaluating the Unoriented Jones Polynomial at $1$}  In this subsection, we explain why the choices of  $(-1)^{-n_-}$ and $(-1)^{\tilde{\lambda}-s_- - \frac12m}$ are important normalization factors for the oriented and unoriented Jones polynomials, $J_L$ and $\widetilde{J}_L$.  Namely, we will show that evaluating either polynomial at $1$ is either $2^\ell$ for an $\ell$-component even link or $0$ if the link is odd.  We start with the following definition and lemma.

\begin{defn}
\label{def:bracketInvariant}
We define a numerical invariant for virtual links by evaluating the bracket polynomial at 1:
$$[L] = \langle L \rangle (1).$$
\end{defn}

Observe that $J_L(1) = (-1)^{-n_-} [L]$ and that $\widetilde{J}_L(1) = (-1)^{\tilde{\lambda}-s_- -\frac12 m}[L]$.

\begin{lem}
\label{lem:CrossingChange}
The usual (oriented) Jones polynomial and unoriented Jones polynomial evaluated at $1$ are each invariant under crossing changes.  That is, if $L_1$ and $L_2$ are oriented virtual links with diagrams that are identical except for a single crossing change, then 
$$J_{L_1}(1) = J_{L_2}(1), \text{ and}$$
$$\widetilde{J}_{L_1}(1) = \widetilde{J}_{L_2}(1).$$
\end{lem}

\begin{proof}
Consider $\widetilde{J}$ first.  The effect of a crossing change on $[L]$ is to multiply by $-1$.  Thus it suffices to show that $\left(\tilde{\lambda} - s_- - \frac12 m \right)$ changes parity under a crossing change.  Suppose $c$ is the crossing to be changed.  If $c$ is a self-crossing, then $\tilde{\lambda}$ and $m$ are unaffected (since they only consider mixed crossings), and $s_-$ changes by $\pm1$.  

If $c$ is a mixed crossing then $s_-$ and $\frac12 m$ remain unchanged.  If $c$ involves two components in the same core of $L$, then the crossing change results in a net change of 1 in the linking number.  If the crossing change is between two components that do not belong to the same core, then we consider the effect of the crossing change on $\widetilde{Lk}(K_i, K_j)$ where $K_i$ and $K_j$ are the two components that cross at c.  When the crossing is switched, $Lk(K_i,K_j)$ will change by 1 which is sufficient to change the parity of $\widetilde{Lk}(K_i, K_j)$ as desired.

The proof for $J$ is simpler:  changing a positive crossing to a negative (or vice-versa) clearly changes the parity of $n_-$ and hence, compensates for the sign change of $[L]$.
\end{proof}

\begin{rem}
One possible extension of $\lambda$ we considered was to use the classical/even $\lambda$ only on the crossings in the even core(s) and ignore all other crossings.   \Cref{lem:CrossingChange} shows why we needed a modified linking number that incorporated every mixed-crossing.
\end{rem}

It is well known that, for a classical link, the usual (oriented) Jones polynomial is, when evaluated at 1, equal to two to the number of components of the link.  It was not known how this result extended to virtual links. The following theorem is a new theorem in the (non-classical) virtual link literature.

\begin{thm}
\label{thm:ComponentCount}
If $L$ is an oriented virtual link with $\ell$ components, then 
\[J_L(1) = \begin{cases} 
      2^{\ell} & \text{if }L \text{ is even} \\
      0 & \text{if  }L\text{ is odd.} 
   \end{cases}
\]
The number, $J_L(1)$, is independent of the choice of orientation.
\end{thm}

We first noticed this result in the category of planar trivalent graphs (cf. \cite{Baldridge2018, BLM2018}), and wondered if a similar result held for virtual links.  The fact that it does was an important motivation behind this current paper and \cite{BKR2019} (See Future Aims, \Cref{section:FutureAims}).

\begin{proof}
Enumerate the components of $L$ by $L = K_1 \cup \ldots \cup K_\ell$.  The proof proceeds by induction on the number of classical crossings in the diagram of $L$.  The theorem is clearly true if there are no classical crossings.

Suppose that there exists some crossing for a diagram of $L$.  By repeated applications of \Cref{lem:CrossingChange} we can assume without loss of generality that every crossing of $L$ is positive, i.e., $n_-=0$.  Let $c$ be a (positive) crossing between $K_i$ and some other component $K_j$ of $L$.  If $L_A$ represents the link $L$ with an $A$-smoothing ($\Across \rightarrow \Asmooth$) at crossing $c$, and $L_B$ represents the link $L$ with a $B$-smoothing ($\Across \rightarrow \Bsmooth$) at crossing $c$, then
\begin{equation}
\label{eqn:Recursion}
[L] = [L_A] - [L_B].
\end{equation}
Both $L_A$ and $L_B$ have one less component than $L$ since $K_i$ is welded to $K_j$ (see \Cref{fig:smoothing}).  
\begin{figure}[H]
\includegraphics[scale = 0.5]{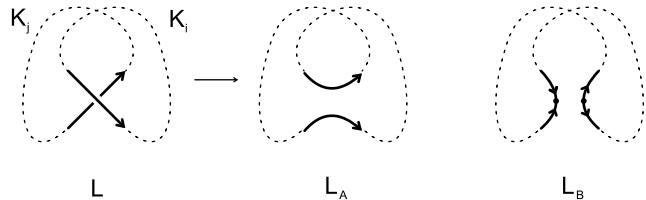}
\caption{Smoothing a crossing.}
\label{fig:smoothing}
\end{figure}
Observe that if $c$ is oriented as shown on the left side of \Cref{fig:smoothing}, then the $A$-smoothing $L_A$ is compatible with this orientation, while the $B$-smoothing $L_B$ must be reoriented as shown in \Cref{fig:reorient}. Let $K_{ij}$ stand for the component in $L_B$ with orientation given by $K_j$ and $\overline{K}_i$, i.e., the part of $K_i$ from $L$ but with the opposite orientation.
\begin{figure}[H]
\includegraphics[scale = 0.5]{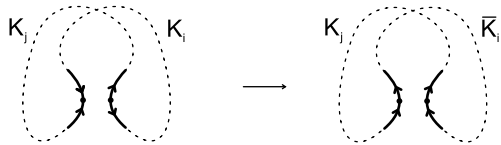}
\caption{Reorienting $L_B$.}
\label{fig:reorient}
\end{figure}
Since $n_-=0$ for $L$, $J_L(1)=[L]$.  Similarly, $J_{L_A}(1)=[L_A]$.  If $K_i$ has $k+1$ classical crossings in $L$, then after reversing the orientation of $K_i$ in $L_B$ to get an orientation for $K_{ij}$, the welded $K_{ij}$ component will have $k$ negative crossings in $L_B$.  Thus, $J_{L_B}(1)=(-1)^k[L_B]$.  Insert these equations for $L$, $L_A$, and $L_B$ into \Cref{eqn:Recursion} to get
\begin{equation}
\label{eqn:Recursion2}
J_L(1) = J_A(1) - (-1)^k J_B(1).
\end{equation}
Note that $(-1)^{k+1} = (-1)^{\pi(K_i)}$. Hence, we can rewrite \Cref{eqn:Recursion2} as $$
J_L(1) = J_A(1) + (-1)^{\pi(K_i)} J_B(1).$$

By induction, $J_A(1)$ must be either $2^{\ell - 1}$ or 0 based on the parity of $L_A$. Similarly, $J_B(1)$ is either $2^{\ell - 1}$ or 0 based on the parity of $L_B$.  Moreover, since $c$ is a crossing between two components of $L$, the parities of $L_A$ and $L_B$ are the same: $\pi(L_A) = \pi(L_B)$.  Thus,
$$J_L(1) = \left(1+(-1)^{\pi(K_i)}\right) 2^{\ell-1} \cdot \pi(L_A).$$

A similar argument applies when the chosen crossing involves only a single component.  In that case, $L_A$ and $L_B$ will have an extra component, but the reasoning remains essentially the same.  Thus, by induction the theorem follows.
\end{proof}

The proof above is for the usual (oriented) Jones polynomial for virtual links, but the result still holds for the unoriented Jones polynomial. There are two key observations needed to see why this is true.  First, by \Cref{thm:ComponentCount}, $[L]=0$ if $L$ is odd.  Thus, we need only check $\tilde{J}_L(1)$ for $L$ even. Second, for an even virtual link, $\tilde{\lambda} = \lambda$ (cf. \Cref{rem:relateLambda}), and $n_- = -\lambda + s_- +\frac12 m$.  Putting these observations together, we get:

\begin{cor}
\label{cor:ComponentCountforUnrientedJones}
If $L$ is an unoriented virtual link with $\ell$ components, then 
\[ \widetilde{J}_L(1) = \begin{cases} 
      2^{\ell} & \text{if }L \text{ is even} \\
      0 & \text{if  }L\text{ is odd.} 
   \end{cases}
\]
\end{cor}

Thus, $\tilde{J}_L(1)=J_L(1)$, which means $\tilde{J}_L$ continues to have the same non-negativity propoerty when evaluated at $1$ as the (oriented) Jones polynomial $J_L$.

\subsection{Other Component Parity Functions and Other Polynomials} \label{subsection:other-parity-functions} 
There are other component parity functions we might have chosen to use in our definition of the modified linking number. Each may be useful in other contexts, though they do not necessarily preserve the properties described  in \Cref{thm:ComponentCount} and \Cref{cor:ComponentCountforUnrientedJones}.

The first option is to redefine the parity function so that only components in the even core $C=C_1$ of $L$ evaluate to $0$:
$$
p_{C}(K_i, K_j) = \left\{
        \begin{array}{ll}
            0 & \quad K_i, K_j\in C \\
            1 & \quad \text{otherwise.}
        \end{array}
    \right.
$$
Using this component parity function in the place of $p$ in the definition of the modified linking number leads to a different $\lambda_C$ and another  polynomial $J_L^C$ for which \Cref{thm:UnorientedJonesISInvariant}  and \Cref{cor:ComponentCountforUnrientedJones} are both true.  While this choice satisfies both of the criteria listed above, the component parity function $p$ is still preferable, as it maximizes the behavior of even links in $L$.

Another option is to define the parity function to be $1$ on every pair of components:
$$
p_{M}(K_i, K_j) = 1 \mbox{ for all $K_i$ and $K_j$}.$$
Using this parity to define a modified linking number gives rise to a polynomial, $J_L^M$, for which \Cref{thm:UnorientedJonesISInvariant} is still true.  However, \Cref{cor:ComponentCountforUnrientedJones} will fail to be true.

The last option we consider is in some sense the simplest, as it does not use $\lambda$, or its variants, at all.  The polynomial we obtain from this choice is
\begin{equation}\label{eqn:J0}
J_L^0(q) = (-1)^{\left(-s_- - \frac12 m \right)} q^{\left(s_+-2s_--\frac12 m\right)} \langle L\rangle.
\end{equation}
The polynomial is an unoriented virtual link invariant in the sense that it does not require an orientation to define it. It can have complex  coefficients, so it does not satisfy the last part of \Cref{thm:UnorientedJonesISInvariant}, nor does it satisfy \Cref{cor:ComponentCountforUnrientedJones}.  Nevertheless, $J_L^0(q)$ has the advantage of being entirely determined by the graded Euler characteristic of the unoriented Khovanov homology described in \Cref{section:UnorientedKhovHom}.

\section{Bracket Homology and Khovanov Homology}
\label{Khovanov-Homology}

In this section, we briefly describe Khovanov homology
along the lines of \cite{Kho,DB1}, and we tell the story so that the gradings and the structure of the differential emerge in a natural way.
This approach to motivating the Khovanov homology uses elements of Khovanov's original approach, Viro's use of enhanced states for the bracket
polynomial \cite{V}, and Bar-Natan's emphasis on tangle cobordisms \cite{DB2}.  

We begin by working without using virtual crossings, and then introduce extra structure and generalize the Khovanov homology to virtual knots and links in the next section.

A key motivating idea involved in defining the Khovanov invariant is the notion of categorification. One would like to {\it categorify} the Kauffman bracket 
$\langle D \rangle$ for a link diagram $D$ of a link $L$. There are many meanings to the term categorify, but here the quest is to find a way to express the link polynomial
as a {\it graded Euler characteristic} $\langle D \rangle = \chi_{q}(H(D))$ for some homology theory associated with $\langle D \rangle$.  In this section, we define a homology theory with this property. Moreover, the homology theory we define can be used with an orientation on $L$  to get the usual Khovanov homology invariant of $L$.  We call this the {\em bracket homology of $D$} and denote it $H(D)$.

The bracket polynomial \cite{K, KNOTS} model for the Jones polynomial \cite{Jones,JO1,JO2,WITT} can be  described by the inductive expansion of unoriented crossings $\Across$ into {\em $A$-smoothings} $\Asmooth$ and {\em $B$-smoothings} $\Bsmooth$ on a link diagram $D$ via
Equation~\ref{eq:bracket-crossing} and Equation~\ref{eq:bracket-disjoint-circ}:
$\langle \Across \rangle=\langle \Asmooth \rangle-q\langle
\Bsmooth \rangle \label{kabr}$ with $\langle \bigcirc\rangle=(q+q^{-1})$.
While the bracket polynomial is often described in a variable $A$, it useful to work with the $q$-variable version in the context of Khovanov homology.

There is a well-known convention for describing the bracket state expansion by {\it enhanced states} where an enhanced state has a label of $1$ or $x$ on each of
its component loops. We then regard the value of the loop $q + q^{-1}$ as
the sum of the value of a circle labeled with a $1$ (the value is
$q$) added to the value of a circle labeled with an $x$ (the value
is $q^{-1}).$ 

To see how the Khovanov grading arises, consider the expansion of the 
bracket polynomial in enhanced states  $\mathfrak{s}$:
$$\langle D \rangle = \sum_{\s} (-1)^{n_{B}(\s)} q^{j(\s)}$$
where $n_{B}(\s)$ is the number of $B$-smoothings  in $\s$, $r(\s)$ is the number of loops in $\s$ labeled $1$ minus the number of loops
labeled $x,$ and $j(\s) = n_{B}(\s) + r(\s)$.
This can be rewritten in the following form:
$$
\langle D \rangle  =  \sum_{i \,,j} (-1)^{i} q^{j} \dim C^{i,j}(D) \label{eq:graded-Euler-Char}
$$
where we define $C^{i,j}(D)$ to be the linear span of the set of enhanced states with $n_{B}(\s) = i$ and $j(\s) = j.$  Then the number of such states is $\dim C^{i,j}(D).$   

We would like to turn the bigraded vector spaces $C^{i,j}$ into a  bigraded complex ($C^{i,j}$, $\partial$) with a differential
$$\partial:C^{i,j}(D)  \longrightarrow C^{i+1,j}(D).$$ 
The differential should increase the {\it homological grading} $i$ by $1$ and preserve the 
{\it quantum grading} $j.$
Then we could write
$$\langle D\rangle = \sum_{j} q^{j} \sum_{i} (-1)^{i} \dim C^{i,j}(D) = \sum_{j} q^{j} \chi\left(C^{\bullet, j}(D)\right),$$
where $\chi\left(C^{\bullet, j}(D)\right)$ is the Euler characteristic of the subcomplex $C^{\bullet, j}(D)$ for a fixed $j$. 
This formula would constitute a type of categorification of the bracket polynomial. 

Below, we
shall see how {\it the original Khovanov differential $\partial$ is uniquely determined by the restriction that $j(\partial \s) = j(\s)$ for each enhanced state
$\s$.} Since $j$ is 
preserved by the differential, these subcomplexes $C^{\bullet, j}$ have their own Euler characteristics and homology. We have
$\chi(H^{\bullet, j}(D)) = \chi(C^{\bullet, j}(D)) $ where $H^{\bullet, j}(D)$ denotes the bracket homology of the complex 
$C^{\bullet, j}(D)$. We can write
$$\langle D \rangle = \sum_{j} q^{j} \chi(H^{\bullet, j}(D)).$$ 
The last formula expresses the bracket polynomial as a {\it graded Euler characteristic}, $\chi_q(H(D))$, of a homology theory associated with the enhanced states of the bracket state summation. This is the desired categorification of the bracket polynomial. Khovanov proved that a gradings-shifted version of this homology theory (using an orientation of the link) is an invariant of oriented knots and links, and that the graded Euler characteristic of the gradings-shifted version is the usual (oriented) Jones polynomial.  Thus, he created a new and stronger invariant than the original Jones polynomial.

To define the differential regard two states as {\it adjacent} if one differs from the other by a single smoothing at some site. Let $(\s,\tau)$ denotes a pair consisting of an enhanced state $\s$ and a site $\tau$ of that state with $\tau$ of type $A$.  Consider all enhanced states $\s'$ obtained from $\s$ by resmoothing $\tau$ from $A$ to $B$ and relabeling (with $1$ or $x$) only those loops that are affected by the resmoothing. Call this set of enhanced states $S'[\s,\tau].$ Define the {\it partial differential} $\partial_{\tau}(\s)$ as a sum over certain elements in $S'[\s,\tau]$ and the differential for the complex by the formula $$\partial(\s) = \sum_{\tau}  (-1)^{c(\s,\tau)} \partial_{\tau}(\s)$$ with the sum over all type $A$ sites $\tau$ in $\s.$  Here $c(\s,\tau)$ denotes the number of $A$-smoothings prior to the $A$-smoothing in $\s$ that is designated by $\tau.$ Priority is defined by an initial choice of order for the crossings in the knot or link diagram.

In Figure~\ref{states}, we indicate the original forms of the states for the bracket (not yet labeled by $1$ or $x$ to specify enhanced states) and their arrangement as a Khovanov category where  the generating morphisms are arrows from one state to another where the domain of the arrow has one more $A$-state than the target of that arrow. In this figure we have assigned an order to the crossings of the knot, and so the reader can see from it how to define the signs for each partial differential in the complex.

\begin{figure}[H]
     \begin{center}
     \begin{tabular}{c}
     \includegraphics[scale=.35]{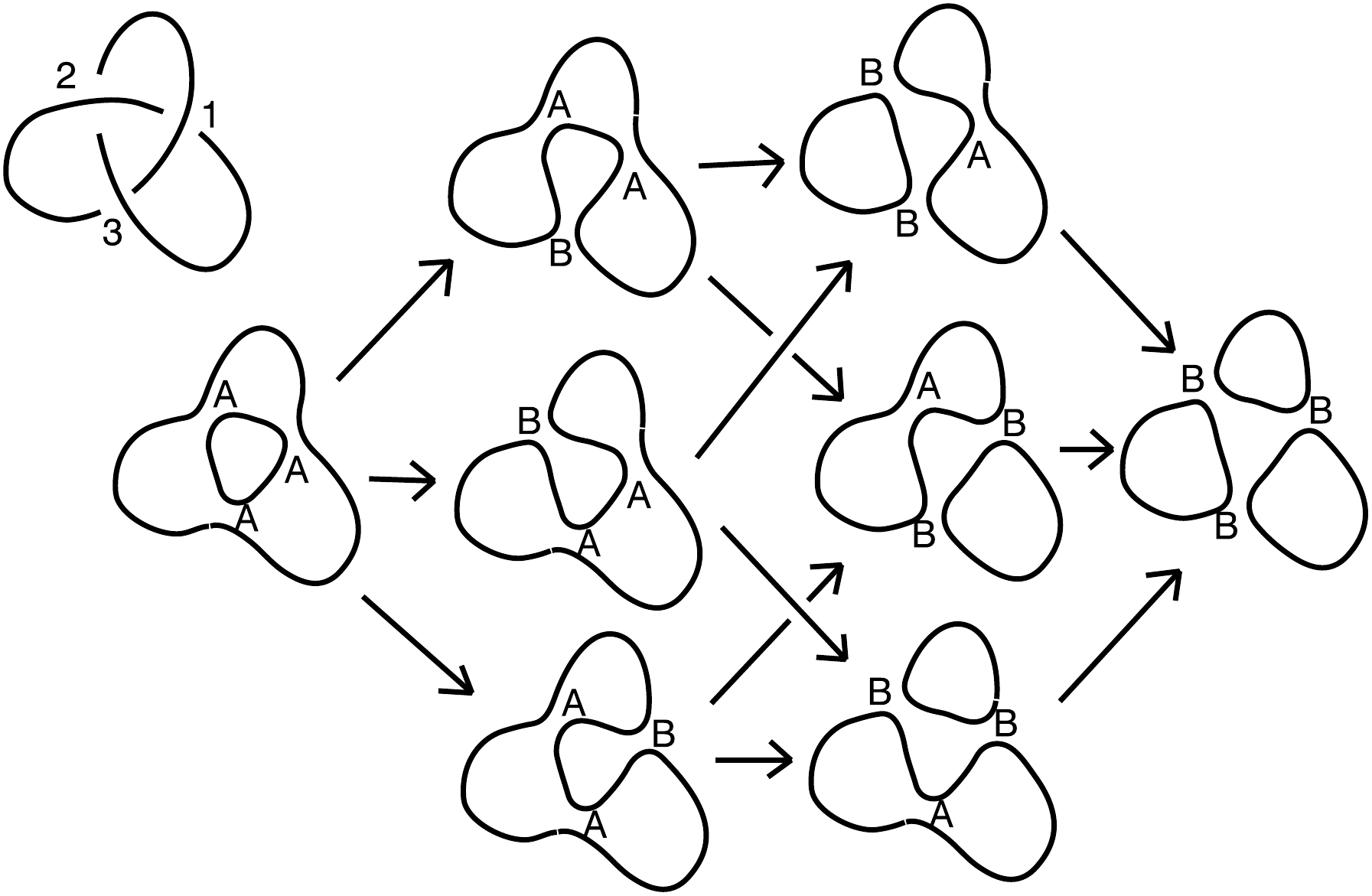}
     \end{tabular}
     \caption{Bracket states and Khovanov complex.}
    \label{states}
\end{center}
\end{figure}

We now explain how to define $\partial_{\tau}(\s)$ so that $j(\s)$ is preserved.  The  form of the partial differential can be described by the  following structure of multiplication and comultiplication on the algebra $V= k[x]/(x^{2})$ where $k = \BZ$ for integral coefficients:
\begin{enumerate}
\item The element $1$ is a multiplicative unit and $x^2 = 0$, and
\item $\Delta(1) = 1 \otimes x + x \otimes 1$ and $\Delta(x) = x \otimes x.$
\end{enumerate}
These rules describe the local relabeling process for loops in an enhanced state. Multiplication corresponds to the case where two loops merge to a single loop, while comultiplication corresponds to the case where one loop bifurcates into two loops. Thus,

\begin{prop} The partial differentials $\partial_{\tau}(\s)$ are uniquely determined by the condition that $j(\s') = j(\s)$ for all $\s'$ involved in the action of the partial differential on the enhanced state $\s.$  \label{prop:partial-determines-grading}
\end{prop}

We briefly describe how to obtain the usual Khovanov homology from the bracket homology.  Let $\{b\}$ denote the degree shift operation that shifts the homogeneous component of a graded vector space in dimension $m$ up to dimension $m + b$.  Similarly, let $[a]$ denote the homological shift operation on chain complexes that shifts the $r$th vector space in a complex to the $(r + a)$th place, with all the differential maps shifted accordingly (cf. \cite{DB1}).  Given an orientation of the link $L$, the crossings in the diagram $D$ of $L$ can be assigned to be positive or negative in the usual way.  If $n_+$ and $n_-$ are the total number of positive and negative crossings (classical, not virtual), we can shift the gradings of the bracket homology $H(D)$ by $[-n_-]$ and $\{n_+-2n_-\}$.  Khovanov proved that this shifted bracket homology was an invariant of the oriented link:

\begin{thm}[Khovanov,\cite{Kho}] Let $D$ be a link diagram of an oriented link $L$. Then
$$Kh(L) \cong H(D)[-n_-]\{n_+-2n_-\}.$$ \label{theorem:khovanov-homology}
\end{thm}

It is this form of the theory that generalizes nicely to an unoriented version of Khovanov homology (see \Cref{section:UnorientedKhovHom}).

\bigskip
There is much more that can be said about the nature of the construction of this section with respect to Frobenius algebras and tangle cobordisms. The partial boundaries can be conceptualized in terms of surface cobordisms. The equality of mixed partials corresponds to topological equivalence of the corresponding surface cobordisms, and to the relationships between Frobenius algebras and the surface cobordism category. The proof of invariance of Khovanov homology with respect to the Reidemeister moves (respecting grading changes) will not be given here. See \cite{Kho,DB1,DB2}. It is remarkable that this version of Khovanov homology is uniquely specified by natural ideas about adjacency of states in the bracket
polynomial.

\section{Bracket Homology of Virtual Links}
\label{section:bracket-homology-for-virtual-links}

In this section, we describe how to define and calculate Khovanov homology for virtual links for arbitrary coefficients.  This section utilizes techniques from \cite{ArbitraryCoeffs} and \cite{DKK}, but with conventions that are amenable to programming.  Such a program has been written, and will be described in a subsequent paper.

Extending Khovanov homology to virtual knots for arbitrary coefficients is complicated by the single cycle smoothing as depicted in the first row and second column of  \Cref{fig:problemsquare}.
We define a map for this smoothing $\eta:V \longrightarrow V$. In order to preserve the quantum grading, we must have that $\eta$ is the zero map.

Consider the following complex in Figure~\ref{fig:problemsquare} arising from the 2-crossing virtual unknot:
\begin{figure}[H]
\begin{center}
    \includegraphics[scale=.16]{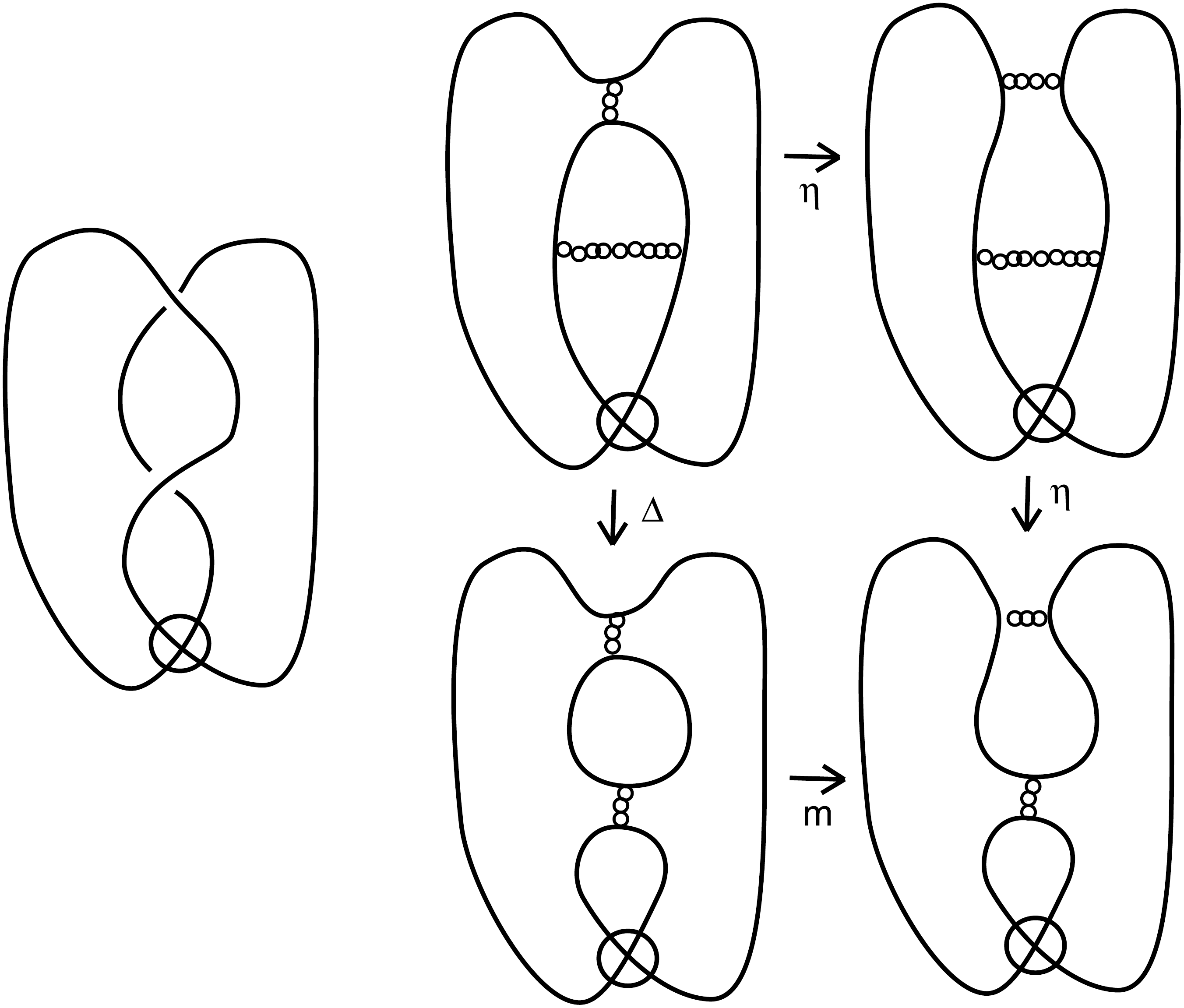}
    \caption{Khovanov complex for the two-crossing virtual unknot.}
\label{fig:problemsquare}
\end{center}
\end{figure}

Composing along the top and right we have $\eta \circ \eta  = 0.$  But composing along the opposite sides we see $$m \circ \Delta(1) = m(1 \otimes x + x \otimes 1) = x + x = 2x.$$ 
Hence the complex does not naturally commute or anti-commute.

When the base ring is $\BZ/2\BZ$ the definition of Khovanov homology given in the previous section goes through unchanged.  Manturov \cite{ArbitraryCoeffs} (see also\cite{ManturovKhovanovZ2},  \cite{ManturovVirtualKnots}) introduced a definition of Khovanov homology for (oriented) virtual knots with arbitrary coefficients.  Dye, Kaestner 
and Kauffman \cite{DKK} reformulated Manturov's definition and gave applications of this theory.  In particular, they found a generalization of the Rasmussen invariant and proved that virtual links with all positive crossings have a generalized four-ball genus equal to the genus of the virtual Seifert spanning surface. The method used in these papers to create an integral version of Khovanov homology involves using the properties of the virtual diagram to make corrections in the local boundary maps so that the individual squares commute.

Our formulation begins with a link diagram with labeled arcs $D$, see \Cref{fig:PD}.  The arcs of each component are labeled numerically in increasing order.  Arc labels change with each pass through a classical crossing.  If a component has only two arcs, stabilize with a Reidemeister 1 move to get more than two labeled arcs. This requirement ensures that the increasing order of the arc labels specifies an orientation of the component by traveling along it (see the right hand picture of \Cref{fig:PD}). Next, a 4-tuple is specified for each classical crossing by recording the arc label of the incoming under-crossing arc and then recording the remaining arc labels in order by going counterclockwise around the crossing. For example, the tuple for the top crossing in the diagram on the left of \Cref{fig:PD} is $(4,3,1,2)$. The set of all 4-tuples is called the {\em PD notation} for the link diagram with labeled arcs $D$ and is denoted $PD(D)$.  For example, the PD notation of the virtual trefoil diagram $vT$ on the left in \Cref{fig:PD} is $PD(vT)=\{ (4,3,1,2), (1,4,2,3)\}$.  For comparison, the PD notation for the link diagram with labeled arcs $D$ on the right is $PD(D)=\{(3,4,4,1),(1,6,2,5),(2,7,3,6),(7,5,8,8)\}$. For a  description of PD notation and the different formats used in the literature, see \cite{KnotAtlas,LinkInfo}.
 
\begin{figure}[H]
\begin{center}
    \includegraphics[scale=0.4]{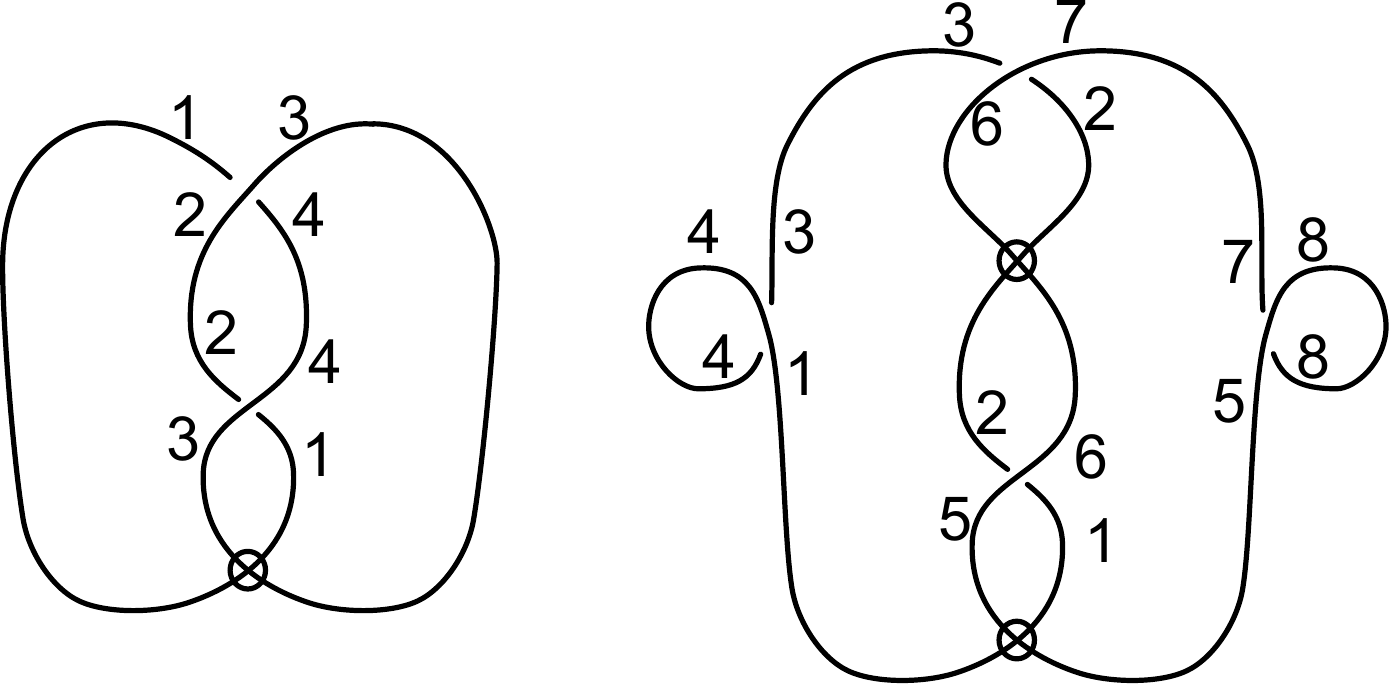}
    \caption{A virtual trefoil and a 2-component virtual link with labeled arcs.}

\label{fig:PD}
\end{center}
\end{figure}

\begin{figure}[H]
     \begin{center}
     \begin{tabular}{c}
     \includegraphics[scale=.8]{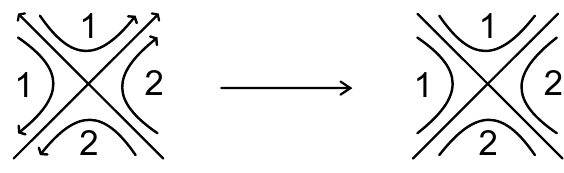}
     \end{tabular}
     \caption{The local order determined by the link orientation.}
    \label{fig:localOrder}
\end{center}
\end{figure}

\begin{figure}[H]
\begin{center}
    \includegraphics[width=.5\textwidth]{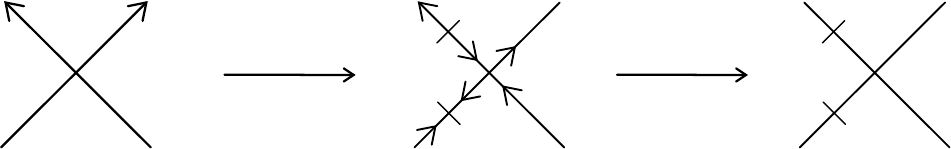}
    \caption{Inserting cut-points at the crossings.}
\label{fig:cutpoints}
\end{center}
\end{figure}

Recall that the formulation of Khovanov homology in \cite{DKK} and \cite{ArbitraryCoeffs} requires a global order for each state and a base-point for each circle in each state.  As we will briefly describe below, using PD notation of a link diagram, $PD(D)$, leads to a natural global order and choice of base point for each circle. Moreover, it also provides a convenient way to determine the locations of cut-points and sign conventions. \\

\noindent {\bf Base points:} Smoothing each classical crossing results in a merger of two pairs of arcs.  With each merger, we label each new arc with the smaller of the two labels.  For example, the zero smoothing of the crossing $(4,3,1,2)$ would result in two new arcs now labeled $1$ and $3$.   After smoothing all classical crossings we are left with a collection of immersed circles in the plane, each labeled by the smallest arc in that circle.  Since each circle in each state is made up of the labeled arcs from $PD(D)$, including the smallest labeled arc, we can place the base point of \cite{DKK} on each circle halfway along this smallest labeled arc. \\
 
\noindent {\bf Global Order:}  A global order is a choice of an ordering for each set of circles in a state. By the process described above for base points, each circle in each state has a label given by the smallest arc that forms that circle.  An ordering for the set of circles in a state is then given by ordering these labeled circles from least to greatest.\\

\noindent {\bf Sign convention for the hypercube: } We follow the approach of \cite{ArbitraryCoeffs} to sprinkle signs on maps that correspond to edges in the hypercube of states.  These signs ensure that each diagram corresponding to a face of the hypercube anti-commutes.  First, an {exterior algebra-like} object (cf. the ordered tensor product in \cite{ArbitraryCoeffs}) is created by taking the wedge product of the labeled circles given by the global order. For example, if a state has circles labeled $\{c_1, c_3, c_7, c_9\}$, then form the object $c_1\wedge c_3 \wedge c_7\wedge c_9$. The sign of the map  corresponding to, say, the merger of circles $c_3$ and $c_9$ in this example is determined by an application of the exterior algebra structure with a comparison to the local order of a crossing, as follows:  Count the number of transpositions to move each circle to the head of the global order to get $- c_3\wedge c_9\wedge c_1\wedge c_7$.  If circle $3$ and $9$ disagree with the local order (cf. \Cref{fig:localOrder}) we transpose them; otherwise we do nothing.  In this example, suppose $c_3$ and $c_9$ disagree with the local order. Thus, we get $c_9\wedge c_3\wedge c_1 \wedge c_7$.  The circles are then merged to form circle $c_3$ (since $3$ is the smaller label, this will be the corresponding labeled circle in the new state and will locate the base-point on that circle).  We then count the number of transpositions required to move $c_3$ to its position determined by the global order: $-c_1\wedge c_3\wedge c_7$.  The total number of transpositions used determines the sign to associate to the merger map $m$. In this example, the number of transpositions is five, so the merger map would be $-m$.

The process for a comultiplication  is similar.  \\

\noindent {\bf Creating a cut-system:}  As in \cite{DKK}, we place cut points on the diagram to create a source-sink orientation.  In order to avoid having to track this extra structure separately (in, say, a computer program), we place cut points at every classical crossing as shown in \Cref{fig:cutpoints}. We place them very near the crossing so as not to confuse where the base points are located with respect to the cut points.  This is  a well-defined cut-system since it is locally well-defined, and away from the classical crossings the source-sink orientation matches that of the orientation of the  link given by the arc labels in $PD(D)$.

\begin{rem}
With the cut-system, base-points and local/global orders so chosen, one may now delete the orientation as in the right hand pictures of  \Cref{fig:localOrder} and \Cref{fig:cutpoints}. From a diagrammatic perspective, an orientation is no longer required to compute the boundary maps as in \Cref{fig:problemsquare1}. However, as described in \Cref{quest:directly-from-diagram}, it is an open question whether or not this can be accomplished without reference to an orientation.
\end{rem}

\noindent {\bf Modifying the local boundary maps:} Algebra to be processed by the local boundary maps is placed initially at the basepoint of each circle in the state. It is then transferred to the site where the map occurs (e.g. joining two circles at that site or splitting one circle into two at that site).  Taking a path along the circle from this basepoint to the site, one will pass either an even number of cut points or an odd number.  If the parity  is even, then both $x$ and $1$ transport to $x$ and $1$ respectively.  If the parity is odd, then $x$ is transported to $-x$ and $1$ is transported to $1$.  The local boundary map is performed (using the signed map described above) on the algebra, and then in the image state, the resulting algebra is transported back to the base point(s).  As above, the path from the site back to the base point will pass through an even or odd number of cut points.  The sign of $x$ and $1$ at the base point is then determined by this parity according to the same rules above.

\begin{ex}
For the two-crossing virtual unknot, we show how these choices of global order, local order, base points, and cut point systems work together.  In Figure~\ref{fig:problemsquare1}, we have illustrated the situation where the top left state is labeled with $1$ and in the left vertical column we have $-\Delta(1) = -1 \otimes x - x \otimes 1.$ We show how the initial element $1$ appears at the base point of the upper left state and how it is transported (as a $1$) to the site for the co-multiplication. The result of the co-multiplication is  $-1 \otimes x - x \otimes 1$  and this is shown at the re-smoothed site. To perform the next local boundary map, we have to transport this algebra to a multiplication site. This result of the co-multiplication is to be transported back to basepoints and then to the new site of multiplication for the next composition of maps. In the figure we illustrate the transport just for $-x \otimes 1.$ At the new site this is transformed to  $x\otimes 1.$ Notice that the $(-x)$ in this transport moves through a single cut-point. We leave it for the reader to see that the transport of $-1 \otimes x$ has even parity for both elements of the tensor product. Thus $-1 \otimes x - x \otimes 1$ is transported to $-1 \otimes x + x \otimes 1$ at the multiplication site. Upon multiplying we have $m(-1 \otimes x + x \otimes 1) = -x + x = 0.$ Thus we now have that the composition of the left and bottom sides of the square is equal to the given zero composition of the right and top sides of the square (which is a composition to two zero single cycle maps). Note that in applying transport for composed maps we can transport directly from one site to another without going back to the base point and then to the second site.  

Note that in Figure~\ref{fig:problemsquare1} one could cancel the two cut points on the arc labeled 2 in the link diagram to simplify by-hand calculations.  However, it is advantageous to keep them in designing a computer program.
\end{ex}

\begin{figure}
\begin{center}
    \includegraphics[scale=.2]{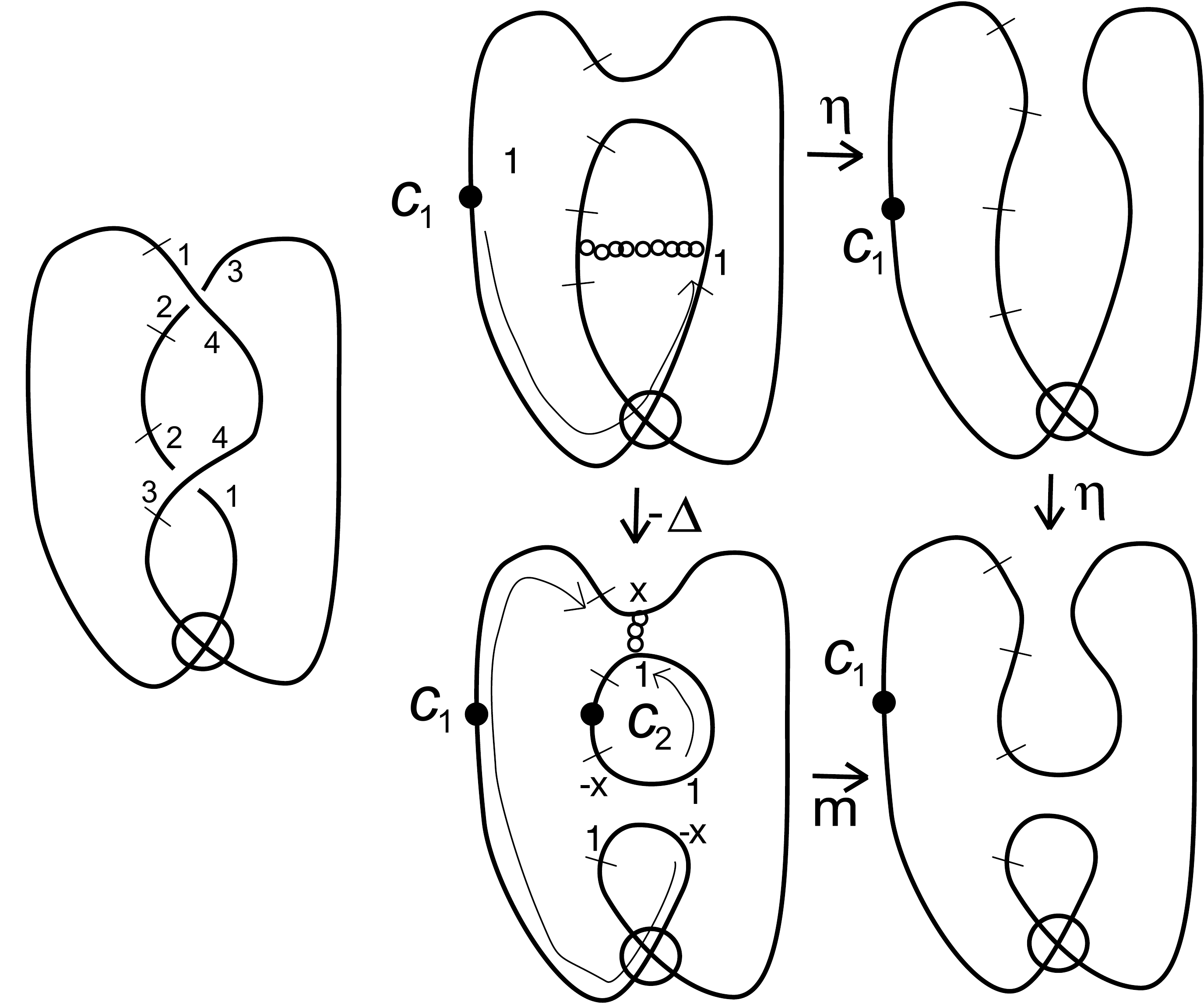}
    \caption{Computing the Khovanov complex for the two-crossing virtual unknot.}
\label{fig:problemsquare1}
\end{center}
\end{figure}

In summary, beginning with a link diagram given in PD notation,  the naming convention for each circle of each state provides a natural global ordering. A consequence of  Manturov's work in \cite{ArbitraryCoeffs} is that one could use any global order (paired with the local order) to specify signs on the maps in the hypercube of states so that each diagram anti-commutes.  This sign convention is different in general from the standard sign convention for Khovanov homology of classical links, but leads to isomorphic theories as explained in \cite{LipSar}.  The cut-system described above is potentially different than the one described in \cite{DKK}, but by \cite{Kamada4} it is cut-point equivalent to it.   A careful analysis of \cite{DKK} shows that the homology we obtain is invariant of the choices made. (For an example of this type of analysis, see \cite{Igor}.) Thus, we have proven:

\begin{thm} \label{thm:MDKK} Given two PD notations, $PD(D)$ and $PD(D')$, for the same  unoriented virtual link $L$, the bracket homology of each will be isomorphic up to a grading shift.  In particular, for some numbers $a$ and $b$, 
$$H(PD(D))\cong H(PD(D'))[a]\{b\}.$$
\end{thm}

The PD notation (as defined in this paper) automatically comes with an orientation since each component is made up of three or more numerically increasing labeled arcs.  Furthermore, bracket homologies are isomorphic up to a shift in gradings even if the orientation from each PD notation is different.

\begin{rem} \label{rem:MDKK} Given a link diagram in PD notation, $PD(D)$, for an oriented virtual link $L$, the PD notation specifies a data structure that includes base points, a global order, a cut point system, and a sign convention for each local boundary map.  This is the data structure needed for computer code to calculate the bracket homology of the diagram.  
\end{rem}

Analysis of the argument leading up to \Cref{thm:MDKK} shows that when $PD(D)$ and $PD(D')$ have the same underlying link diagram, but with different arc labels (possibly leading to different orientations), the homologies are isomorphic without any grading shift, i.e., $H(PD(D))\cong H(PD(D'))$.

Because the arcs can be labeled in an unoriented link diagram so as to induce any orientation of the link without changing the homology, the above says that the bracket homology is an invariant of the unoriented link diagram itself.  Thus, we may make the the following definition and work with link diagrams instead of PD notation from now on.

\begin{defn}
\label{defn:bracketHomDiag}
Given an unoriented link diagram D, label the arcs in order to produce a PD notation, $PD(D)$.  The chain complex and bracket homology of $D$ is defined to be the chain complex and bracket homology of $PD(D)$, respectively.  For example,
$$H(D) \coloneqq H(PD(D)).$$
\end{defn}

\medskip

Since the bracket homology is invariant under changes in the diagram's orientation,  \Cref{defn:bracketHomDiag} also makes sense if we begin with an oriented link diagram. In that case, the bracket homology theory we obtain is isomorphic to the Khovanov homology after an appropriate grading shift.

\begin{cor} \label{cor:MDKK} Let $D$ be a link diagram of an oriented virtual link $L$. If $n_+$ and $n_-$ are the number of positive and negative crossings respectively in $D$, then
$$H(D)[-n_-]\{n_+-2n_-\} \cong Kh(L).$$ 
\end{cor}

Note the similarity between \Cref{cor:MDKK} and \Cref{theorem:khovanov-homology}.  The bracket homology for virtual links defined in this section is a natural extension of the bracket homology for classical links described in \Cref{Khovanov-Homology}.  In the classical case, the two homologies are isomorphic.  Other definitions for virtual Khovanov homology have been given by \cite{Tubbenhauer} and \cite{Rushworth}. Each of these definitions give different solutions to handling the difficult diagram $m\circ \Delta = \eta \circ \eta$ discussed above.  

\begin{rem}While we focus above on invariance under changes of arc labelings and orientations, it should be said that the PD notation expresses each virtual link entirely without mention of virtual crossings. This is appropriate since the virtual crossings are an artifact of the planar representation of a virtual link. This means that our definition of bracket homology for virtual links, expressed in PD notation is also independent of the pattern of virtual crossings in any given planar diagram. Nevertheless, an abstract link diagram (a surface supporting the link, see \cite{DKK}) can be formed from the data in the PD notation and the genus of it can be calculated directly from this data.

\end{rem}

\section{Unoriented Khovanov Homology}
\label{section:UnorientedKhovHom}

We are now ready to describe the unoriented Khovanov homology in terms of the bracket homology and a grading shift.  Let $L$ be a virtual link and $D$ a diagram of the link (for use in a computer program, $D$ would be given in PD notation).  Using \Cref{thm:MDKK} and \Cref{defn:bracketHomDiag}, the bracket homology $H(D)$ is an invariant of the link up to grading shifts, i.e., given any two diagrams $D_1$ and $D_2$ of the same link, there are numbers $a$ and $b$ such that  $H(D_1) \cong H(D_2)[a]\{b\}$.  While the bracket homology requires the choice of an orientation for computation (though it is independent of that choice), the grading shifts $(-s_--\frac12 m)$ and $(s_+-2s_- - \frac12 m)$ needed to define the unoriented Khovanov homology do not require such a choice (see the Introduction and \Cref{sec:unoriented-jones-poly-classical}).  

 For a diagram $D$ of a link $L$, shift the bracket chain complex $(C(D), \partial)$ to get the {\em unoriented  Khovanov chain complex}: 
 \begin{equation}
 \tilde{C}(D)=C(D)[-s_- - \frac12 m]\{s_+ - 2s_- - \frac12 m \}. \label{eq:bracket-chain-complex}
 \end{equation}
We do not include $\lambda$, as $\lambda$ can jump by an even integer by choosing a different orientation. Let $\widetilde{Kh}(D)$ be the homology of this chain complex.

\begin{thm} Let $L$ be an unoriented virtual link, and $D_1$ and $D_2$ be two diagrams of $L$ that are equivalent except they differ by one virtual Reidemeister move (virtual or classical).  Then $\widetilde{Kh}(D_1)\cong \widetilde{Kh}(D_2)$.   \label{thm:unoriented-link-invariant-homology}
\end{thm}

\begin{proof} For a diagram $D$ of a virtual link $L$, the homology of the unoriented Khovanov complex $\tilde{C}(D)$ is equal to $H(D)[-s_- - \frac12 m]\{s_+ - 2s_- - \frac12 m\}$.  Since $H(D_1)$ and $H(D_2)$ are isomorphic up to a gradings shift (cf. \Cref{thm:MDKK} and \Cref{defn:bracketHomDiag}), we only need to show that if $\widetilde{Kh}(D_1)\cong \widetilde{Kh}(D_2)[a]\{b\}$, then $a=b=0$.  Since $D_1$ and $D_2$ differ by a single Reidemeister move, we check each type of move.  Virtual moves V1, V2, V3 and VM (cf. \Cref{Fig:vrms}) do not effect the enhanced states, $s_+$, $s_-$, or $m$.  Hence $\widetilde{Kh}(D_1)=\widetilde{Kh}(D_2)$ in that case.  For the first classical Reidemeister move and the second classical Reidemeister move performed on the same component, the terms $-s_-$ and $s_+-2s_-$ shift the bracket homology by the same number as in the proof of the oriented Khovanov homology.  For the second classical Reidemeister move performed between two different components,  $H(D_1) \cong H(D_2)[1]\{1\}$ for the move that removes two mixed-crossings in $D_1$.  In this case, the term $-\frac12 m$ in both grading shifts compensates for this change in grading.  Finally, the third classical Reidemeister move does not change $s_+, s_-$ or $m$, and it induces an isomorphism $H(D_1)\cong H(D_2)$.  In each case, $a=b=0$, which was to be shown.
\end{proof}

A corollary of  this theorem is that $\widetilde{Kh}(D)$ is isomorphic for every diagram $D$ of $L$.  Thus, $\widetilde{Kh}(D)$ is an invariant of $L$.  Therefore, we  define:

\begin{defn} Let $L$ be an unoriented virtual link and $D$ be a diagram of $L$.  The {\em unoriented  Khovanov homology of $L$}, denoted $\widetilde{Kh}(L)$, is the homology of the complex $\tilde{C}(D)$. 
\end{defn}

\begin{rem} The gradings for the chain complex $(C(D), \partial)$ are always integer valued, but the gradings of the shifted unoriented Khovanov chain complex $(\tilde{C}(D),\partial)$ may be half-integer valued.     
\end{rem}

For classical links, the gradings of $\tilde{C}(D)$ are always integer valued---there are always an even number of mixed-crossings in a classical link, thus $\frac12 m$ is always an integer.  One might conjecture that the same is true for even links since each component has an even number of mixed-crossings.  This is not true, however, as the even virtualized Borromean rings in \Cref{fig:borromean-rings} shows:

\begin{figure}[H]
\begin{center}
    \includegraphics[scale=.8]{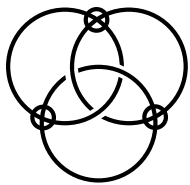}
    \caption{An even virtual Borromean Rings with an odd number of mixed-crossings.}
\label{fig:borromean-rings}
\end{center}
\end{figure}

Since the homology $\widetilde{Kh}(L)$ can be graded by half-integers, we must extend the usual integral grading to the additive group $\frac12\BZ$.   The graded Euler characteristic for unoriented Khovanov homology is then $$\chi_q(\widetilde{Kh}(L)) = \sum_{i,j\in \frac12\BZ} (-1)^i q^j \dim(\widetilde{Kh}^{i,j}(L)).$$

Everything in the formula above continues to make sense if we choose the standard square root of $-1$, i.e., $(-1)^\frac12=  \rm{i}$. The graded Euler characteristic of the unoriented Khovanov homology is a polynomial that may have imaginary valued coefficients.  Therefore, evaluating the graded Euler characteristic at $1$ is of the form ${\rm i}^k\cdot 2^\ell$ for an even virtual link and $0$ for an odd link.  We could have defined the unoriented Jones polynomial as this graded Euler characteristic, which would yield $J_L^0$ (cf. \Cref{eqn:J0}).  For some purposes this could be a reasonable thing to do.  However, normalizing the polynomial so that it evaluates to $2^\ell$ (even) or $0$ (odd) is desirable from the standpoint of matching and generalizing already known theorems in classical link theory. The main motivation behind working with the core and mantle of \Cref{Subsec:EvenCore} was to establish an overall normalization that makes the unoriented Jones polynomial have integer-valued coefficients and evaluates like the oriented Jones polynomial.  It is the reason for the extra $(-1)^{\tilde{\lambda}}$ in \Cref{eqn:unorientedVirtualJones}.  Thus, up to a well defined ``sign,'' the unoriented Khovanov homology categorifies the unoriented Jones polynomial:

\begin{thm}
Let $L$ be a virtual link.  Then 
$$\tilde{J}_L(q) = (-1)^{\tilde{\lambda}}\chi_q(\widetilde{Kh}(L)),$$
where $\tilde{\lambda}\in \frac12 \BZ$. \label{thm:graded-Euler-char}
\end{thm}

The complex number $(-1)^{\tilde{\lambda}}$, i.e., the ``sign,'' is calculated by choosing an orientation of the virtual link diagram of $L$, but once computed, the result is independent of the choice of that orientation by \Cref{thm:lambdaInvariant3}.  

\begin{rem}
The correction provided by $(-1)^{\tilde{\lambda}}$ could be incorporated into the homological grading of $\tilde{C}(D)$.  Define a function $\tilde{l}$ of $\tilde{\lambda}$ to the set  $\{0, \frac12, 1, \frac32\}$ by
$$ \tilde{l} = \left\{ \begin{array}{cc} 
0 & \mbox{\rm  if } 2\tilde{\lambda} \equiv 0  \ (\mbox{\rm mod } 4)\\
\frac12 & \mbox{\rm if } 2 \tilde{\lambda} \equiv 1\ (\mbox{\rm mod } 4)\\
1 & \mbox{\rm if } 2 \tilde{\lambda} \equiv 2\ (\mbox{\rm mod } 4)\\
\frac32 & \mbox{\rm if } 2 \tilde{\lambda} \equiv 3 \ (\mbox{\rm mod } 4)\\
\end{array}\right.$$
The value $\tilde{l}$ is the same number for any orientation by \Cref{thm:lambdaInvariant3}. Replacing $(-s_- - \frac12 m)$ with $(\tilde{l} -s_- -\frac12m)$ in \Cref{eq:bracket-chain-complex} gives an unoriented Khovanov homology whose graded Euler characteristic is  $\tilde{J}_L$.   
\end{rem}

 \section{Lee Homology of Unoriented Links}
 
Lee \cite{Lee} makes another homological invariant of knots and links by using a different
Frobenius algebra. She takes the algebra ${\mathcal A} = k[x]/(x^{2} -1)$  with
$$x^2 = 1,$$
$$\Delta(1) = 1 \otimes x + x \otimes 1,$$
$$\Delta(x) = x \otimes x + 1 \otimes 1,$$
$$\epsilon(x) = 1,$$
$$\epsilon(1) = 0.$$
This can be used to define a differential $\partial'$ and a link homology theory that is distinct from Khovanov homology. In this theory, the quantum
grading $j$ is not preserved, but  {\em one can  use $j$ to filter the chain complex for the Lee homology.} The result is a spectral sequence that starts from Khovanov homology and converges to Lee homology.

We can extend Lee's Frobenius algebra to virtual links to get a bracket complex for Lee theory as follows.  The involution defined in \Cref{section:bracket-homology-for-virtual-links} that takes $x \mapsto -x$ as it is transported through a cut point leads to a well-defined, filtered bracket chain complex, $(C'(D),\partial')$, for the algebra $\mathcal{A}$.   After shifting overall by the unoriented gradings-shifts presented in this paper, we get a Lee theory for a link that does not require a choice of orientation to define the homology:

\begin{thm}
\label{MainThm:LeeHomology}
Let $L$ be an unoriented virtual link and $D$ be any virtual diagram of $L$.  The {\em unoriented Lee Homology} $Kh'(L)$, i.e.,  the homology of the chain complex $$(C'(D)[-s_- -\frac12m]\{s_+-2s_- - \frac12m\}, \partial'),$$
is an invariant of the link $L$.
\end{thm}

The usual (oriented) Lee homology is simple for classical links. One has that the dimension of the Lee homology is equal to 
$2^{\ell}$ where $\ell$ is the number of components of the link $L$ (cf. \Cref{thm:ComponentCount}).
Up to homotopy, Lee's homology has a vanishing differential, and the complex behaves well under link concondance. In his paper \cite{DB3}, Dror BarNatan remarks, ``In a beautiful article Eun Soo Lee
introduced a second differential on the Khovanov complex of a knot (or link) and showed that 
the resulting (double) complex has non-interesting homology. This is a very interesting result."
Rasmussen \cite{JR} uses Lee's result to define invariants of links that give lower bounds for the 
four-ball genus, and determine it for torus knots. Rasmussen's invariant gives an (elementary) proof of a conjecture of Milnor that had been previously shown using gauge theory by Kronheimer and Mrowka \cite{KM1,KM2}.

In \cite{DKK}, Lee homology was generalized to virtual knots and links.  Applications of it to unoriented links can be articulated again with the methods of the present paper. We will carry this out in detail in a future paper.

\section{Future Aims}
\label{section:FutureAims}

This paper has been devoted to formulating an unoriented version of the Jones polynomial (via a normalization of the Kauffman bracket polynomial) and a corresponding version of 
Khovanov homology for virtual knots and links that is an unoriented link invariant. We intend the present paper as a basis for further research and wish to make the following points about future work.\\

\begin{enumerate}

 \item The dependence of the invariant on a choice of orientations is useful in certain contexts.  For example, orientations are useful in the context of oriented cobordisms.  An invariant of the underlying link is useful as well and may inform on unoriented cobordisms. We will explore the unoriented version of Lee homology for virtual links described above  and its applications to cobordisms, genus, and Rasmussen invariants in future research.

\item   This paper grew out of a search for an invariant in a different context: the 2-factor polynomial for ribbon graphs.  A ribbon graph $G$ with a perfect matching $M$ can be made to behave like a knot by orienting the cycles in $G\setminus M$ (see \cite{Baldridge2018}).  However, to define an invariant of  a ribbon graph that is independent of the choice of perfect matchings of the graph {\em and} orientations on the complementary cycles required an ``orientation free'' invariant.  We explore this idea in  \cite{BKR2019}.

\item There is a computer program for virtual homology as formulated by Tubbenhauer (\verb+http://www.dtubbenhauer.com/vKh.html+) and a computer program for Khovanov homology for classical links available at Dror Bar Natan's website (\verb+http://katlas.org/wiki/Khovanov_Homology+).  The authors together with Heather Dye and Aaron Kaestner have generalized Bar Natan's program to calculate the  homology theories discussed in this paper using Theorem~\ref{thm:MDKK}. This will appear in a subsequent paper.\\

\end{enumerate}

At the present time, we know remarkably little about virtual Khovanov homology. It is our intent that this situation will begin to change with the tools developed in this paper.

\end{document}